\documentclass[11pt,a4paper]{article}
\usepackage{amsmath,amssymb,amsthm,mathrsfs}
\usepackage[english]{babel}

\newcommand{\al}{\alpha}
\newcommand{\Om}{\Omega}

\newcommand{\del}{\delta}
\newcommand{\Del}{\Delta}
\newcommand{\ep}{\epsilon}
\newcommand{\lam}{\lambda}
\newcommand{\Lam}{\Lambda}
\newcommand{\Gam}{\Gamma}
\newcommand{\gam}{\gamma}
\newcommand{\ffi}{\varphi}
\newcommand{\sig}{\sigma}

\newcommand{\hac}{\mathcal{H}}
\newcommand{\hact}{\mathcal{H}_T}
\newcommand{\cinf}{\mathcal{C}^\infty}
\newcommand{\N}{\mathbb{N}}
\newcommand{\re}{\mathbb{R}}
\newcommand{\red}{\mathbb{R}^{d}}

\newcommand{\tf}{\mathcal{F}}
\newcommand{\ca}{\mathcal{A}}
\newcommand{\cg}{\mathcal{G}}

\newcommand{\ci}{\mathcal{I}}
\newcommand{\cs}{\mathcal{S}}

\newcommand{\beq}{\begin{equation}}
\newcommand{\eeq}{\end{equation}}

\renewcommand{\mathcal}{\mathscr}

\renewcommand{\P}{\mathrm{P}}
\newcommand {\E}{{\mathrm E}}

\newcommand{\R}{\mathbb{R}}
\newcommand{\D}{\mathbb{D}}

\newtheorem{stat}{Statement}[section]
\newtheorem{prop}[stat]{Proposition}

\newtheorem{thm}[stat]{Theorem}
\newtheorem{lem}[stat]{Lemma}
\newtheorem{exa}[stat]{Example}
\newtheorem{defi}[stat]{Definition}
\theoremstyle{definition}
\newtheorem{remark}[stat]{Remark}
\numberwithin{equation}{section}

\begin{document}
\title{ \bf Gaussian estimates for the density of the non-linear stochastic heat equation in any space dimension}

\author{Eulalia Nualart$^{1}$ and Llu\'is Quer-Sardanyons$^{2}$}

\date{}

\maketitle

\footnotetext[1]{Institut Galil\'ee, Universit\'e
Paris 13, 93430 Villetaneuse, France. Email: \texttt{eulalia@nualart.es}, http://nualart.es}

\footnotetext[2]{{\emph{Corresponding author.}} Departament de Matem\`atiques, Universitat Aut\`onoma de Barcelona, 08193 Bellaterra (Barcelona), Spain. Tel. +34 935814542, Fax +34 935812790. Email: \texttt{quer@mat.uab.cat}}

\maketitle
\begin{abstract}
In this paper, we establish lower and upper Gaussian bounds for the probability density of the mild solution to the stochastic
heat equation with multiplicative noise and in any space dimension. The driving perturbation is a Gaussian noise which is white in time with some spatially homogeneous covariance. These estimates are obtained using tools of the Malliavin calculus. 
The most challenging part is the lower bound, which is obtained by adapting a general method developed by Kohatsu-Higa to the underlying spatially homogeneous Gaussian setting. Both lower and upper estimates have the same form: a Gaussian density with a variance which is equal to that of the mild solution of the corresponding linear equation with additive noise. 
\end{abstract}

\vskip 1,5cm {\it \noindent AMS 2010 Subject Classification:} 60H15, 60H07. \vskip 10pt

\noindent {\it Key words:} Gaussian density estimates; Malliavin calculus; spatially homogeneous Gaussian noise; stochastic heat equation. \vskip 4cm \pagebreak

\section{Introduction and main result}

In this paper, we aim to establish Gaussian lower and upper estimates for the probability density of the solution to the
following stochastic heat equation in $\red$:
\begin{equation} \label{equa1}
\frac{\partial u}{\partial t}(t,x)-\Del u(t,x)= b(u(t,x)) + \sigma(u(t,x)) \dot{W}(t,x), \; \; (t,x)\in [0,T]\times \R^d,
\end{equation}
with initial condition $u(0,x)=u_0(x)$, $x\in \red$. Here, $T>0$ stands for a fixed time horizon, the coefficients $\sigma,b:\R \rightarrow  \R$ are smooth functions and $u_0:\R^d \mapsto \R$ is assumed to be measurable and bounded. As far as the driving perturbation is concerned, we will assume that $\dot{W}(t,x)$ is a Gaussian noise which is white in time and has a spatially  homogeneous covariance.  This can be formally written as:
\begin{equation}\label{eq:corr}
\E \, [\dot{W}(t,x) \dot{W}(s,y)]= \delta(t-s) \Lam(x-y),
\end{equation}
where $\delta$ denotes the Dirac delta function at zero and $\Lam$ is some tempered distribution on $\red$ which is the Fourier transform of a non-negative tempered measure $\mu$ on $\R^d$ (the rigorous definition of this Gaussian noise will be given in Section \ref{sec:malliavin}). The measure $\mu$ is usually called the {\it{spectral measure}} of the noise $W$.

The solution to equation (\ref{equa1}) will be understood in the {\it{mild}} sense, as follows. 
Let $(\mathcal{F}_t)_{t \geq 0}$ denote the filtration generated by the spatially homogeneous noise $W$ (see again Section \ref{sec:malliavin} for its precise definition). 
We say that an $\tf_t$-adapted process $\{u(t,x),\; (t,x)\in [0,T]\times \red\}$ solves (\ref{equa1}) if it satisfies:
\begin{equation} \begin{split}
u(t,x)= (\Gam(t) \ast u_0)(x)&+ \int_0^t \int_{\red} \Gam(t-s,x-y)\sig(u(s,y)) W(ds,dy)  \\
& + \int_0^t \int_{\red} \Gam(t-s,x-y) b(u(s,y)) \, dy ds,
\label{eq:11}
\end{split}
\end{equation}
where $*$ is the standard convolution product in $\red$, and $\Gam$ denotes the fundamental solution associated to the heat equation on $\red$, that is, the Gaussian kernel of variance $2t$:
$\Gam(t,x)=(4\pi t)^{-\frac d2} \exp{\left(-\frac{\Vert x \Vert^2}{4t}\right)}$,  for $(t,x)\in \re_+\times \red$. Note that the stochastic integral on the right-hand side of (\ref{eq:11}) can be understood either in the sense 
of Walsh \cite{Walsh:86}, or using the further extension of Dalang \cite{Dalang:99} (see also \cite{Nualart:07,Dalang-Quer} for another equivalent approach). Indeed, \cite[Theorem 13]{Dalang:99} and \cite[Theorem 4.3]{Dalang-Quer} imply that equation (\ref{eq:11}) has a unique solution which is $L^2$-continuous and satisfies, for all $p\geq 1$:
\[
\sup_{(t,x)\in [0,T]\times \red} E \left[|u(t,x)|^p\right]<+\infty.
\]
Let us point out that, in the above-mentioned results, the fundamental solution $\Gam$ and the noise $W$ are related as follows:
\begin{equation}
\label{hyp1}
\Phi(T):=\int_0^T \int_{\R^d} \vert \mathcal{F}\Gamma(t) (\xi) \vert^2 \mu(d\xi) dt<+\infty.
\end{equation}
This quantity measures the variance of the stochastic integral in (\ref{eq:11}) (indeed, it is the variance itself when $\sig\equiv 1$), therefore it is natural that $\Phi(t)$ will play an important role in the Gaussian lower and upper bounds for the density of the random variable $u(t,x)$. Moreover, it has been proved in \cite[Example 2]{Dalang:99} that condition (\ref{hyp1}) is satisfied if and only if: 
\beq
\int_{\red}\frac{1}{1+\Vert \xi \Vert^2} \, \mu(d\xi)<+\infty.
\label{eq:111}
\eeq

We also remark that the stochastic heat equation (\ref{equa1}) has also been studied in the more abstract framework of Da Prato and Zabczyk \cite{Daprato:92} and, in this sense, we refer the reader to \cite{Peszat-Zabczyk} and references therein. Nevertheless, in the case of our spatially homogeneous noise, the solution in that more abstract setting could be obtained from the solution to equation (\ref{eq:11}) (see \cite[Sec. 4.5]{Dalang-Quer}).

The techniques of the Malliavin calculus have been applied to equation (\ref{eq:11}) in the papers \cite{Marquez,Nualart:07}. Precisely, \cite[Theorem 6.2]{Nualart:07} states that, 
if the coefficients $b$ and $\sigma$ are $\mathcal{C}^{\infty}$-functions with bounded derivatives of order greater than or equal to one, the diffusion coefficient is non-degenerate 
(i.e. $\vert \sigma(z) \vert \geq c >0$ for all $z\in \re$), and (\ref{eq:111}) is satisfied, 
then for each $(t,x) \in (0,T] \times \R^d$, the random variable $u(t,x)$ has a $\mathcal{C}^{\infty}$ density $p_{t,x}$ 
(see also Theorem \ref{thm:smooth-density} below). 
Moreover, in the recent paper \cite{ENualart:10}, the strict positivity of this density has been established under a $\mathcal{C}^{1}$-condition on the density and the additional condition of $\sigma$ being bounded.

Our aim in this paper is to go a step further and prove the following theorem:
 
\begin{thm} \label{maint}
Assume that condition \textnormal{(\ref{eq:111})} is satisfied and $\sigma, b \in \mathcal{C}_b^{\infty}(\R)$ ($\mathcal{C}^\infty$, bounded and bounded derivatives).
Moreover, suppose that $\vert \sigma(z) \vert \geq c >0$, for all $z \in \R$. Then, for every $(t,x)\in (0,T]\times \red$, the law of the random variable $u(t,x)$ has a $\cinf$ density $p_{t,x}$ satisfying, for all $y \in \R$:
\begin{equation*}
C_1 \Phi(t)^{-1/2} \exp \biggl( -\frac{|y-F_0|^2}{C_2 \Phi(t)}\biggr) \leq
p_{t,x}(y) \leq c_1 \Phi(t)^{-1/2} \exp \biggl( -\frac{(\vert y-F_0\vert-c_3 T)^2}{c_2 \Phi(t)}\biggr),
\end{equation*}
where $F_0=(\Gam(t) \ast u_0)(x)$
and $c_1,c_2, c_3, C_1, C_2$ are positive constants that only depend on $T$, $\sigma$ and $b$.
\end{thm}

One of the interests of these type of bounds is to understand the behavior of the density when $y$ is large and $t$ is small. In both cases,
one obtains the same upper and lower behavior for the density, that is, a Gaussian density with a variance which is equal to that of the stochastic integral term in the mild form of the linear equation. We observe that this variance does not depend on $x$ due to the spatially homogeneous structure of the noise.

In order to prove our main result, we will apply the techniques of the Malliavin calculus, for which we refer the reader to \cite{Nualart:06} and \cite{Marta}.  
Obtaining lower and upper Gaussian bounds for solutions to non-linear stochastic equations using the Malliavin calculus has been a current subject of research in the last twenty year. 
Precisely, the expression for the density arising from the integration-by-parts formula of the Malliavin calculus provides a direct way for obtaining an upper 
Gaussian-type bound for the density. Indeed, ones applies H\"older's inequality, and then combines the exponential martingale inequality together with estimates for the 
Malliavin norms of the derivative and the Malliavin matrix. This is a well-known method that has been applied in many situations (see for instance \cite{Guerin,Dalang:09}).
We will also apply this technique to the density of our stochastic heat equation in order to show the upper bound in Theorem \ref{maint} (see Section 5). 

On the other hand, to obtain Gaussian lower bounds for some classes of Wiener functionals turns out to be a more difficult and challenging issue. 
In this sense, the pioneering work is the article by Kusuoka and Stroock \cite{Kusuoka:87}, where the techniques of the Malliavin calculus have been applied to obtain a Gaussian lower estimate for the density 
of a uniformly hypoelliptic diffusion whose drift is a smooth combination of its diffusion coefficient. Later on, in \cite{Kohatsu:03},
Kohatsu-Higa took some of Kusuoka and Stroock's ideas and constructed a general method to prove that the density of a multidimensional functional of the Wiener sheet in $[0,T] \times \R^d$ 
admits a Gaussian-type lower bound. Then, still in \cite{Kohatsu:03}, the author applies his method to a one-dimensional stochastic heat equation in $[0,1]$ driven by the space-time white noise, 
and obtains a lower estimate for the density of the form:
\[
C_1\, t^{-\frac 14 } \exp\left(-\frac{|y-F_0|^2}{C_2\, t^{\frac 12}}\right),
\]
where $F_0$ denotes the contribution of the initial condition. This is the bound we would get if in equation (\ref{equa1}) we let $d=1$ and $\dot{W}$ be the space-time 
white noise. Indeed, this case corresponds to take $\Lam=\del$ in (\ref{eq:corr}), therefore the spectral measure $\mu$ is the Lebesgue measure on $\red$ and $\Phi(t)=C\, t^{\frac 12}$. 
As we will explain below, in the present paper we will use Kohatsu-Higa's method adapted to our spatially homogeneous Gaussian setting. 
The same author applied his method in \cite{Kohatsu:03b} to obtain Gaussian lower bounds for the density of uniformly elliptic non-homogeneous diffusions.
Another important case to which the method of \cite{Kohatsu:03} has been applied corresponds to a two-dimensional diffusion, 
which is equivalent to deal with a reduced stochastic wave equation in spatial dimension one, a problem which has been tackled in 
\cite{Dalang:04}. Moreover, let us also mention that the ideas of \cite{Kohatsu:03} have been further developed by Bally in \cite{Bally:06} in order to deal with more general diffusion processes, 
namely locally elliptic It\^o processes, and this has been applied for instance in \cite{Guerin}. Eventually, in \cite{Bally-Kohatsu} Bally and Kohatsu-Higa have recently combined their ideas in order to obtain lower bounds for the density a class of hypoelliptic two-dimensional diffusions, with some applications to mathematical finance.  

The increasing interest in finding Gaussian lower estimates for Wiener functionals has produced three very recent new approaches, 
all based again on Malliavin calculus techniques. First, in \cite{Nourdin:09} the authors provide sufficient conditions on a random variable in the Wiener space such that its density exists an admits an explicit formula, from which one can study possible Gaussian lower and upper bounds. This result has been applied in \cite{Nualart:09,Nualart:10} to our stochastic heat equation (\ref{equa1}) in the case where $\sig\equiv 1$. Precisely, \cite[Theorem 1 and Example 8]{Nualart:10} imply that, if $b$ is of class $\mathcal{C}^1$ with bounded derivative and condition (\ref{eq:111}) is fulfilled, then, for sufficiently small $t$, $u(t,x)$ has a density $p_{t,x}$ satisfying, for almost all $z\in \re$:
\[
\frac{E|u(t,x)-M_{t,x}|}{C_2\, \Phi(t)} \exp\left( -\frac{|z-M_{t,x}|^2}{C_1\, \Phi(t)} \right)\leq p_{t,x}(z) \leq 
\frac{E|u(t,x)-M_{t,x}|}{C_1\, \Phi(t)} \exp\left( -\frac{|z-M_{t,x}|^2}{C_2\, \Phi(t)} \right),
\]
where $M_{t,x}=E(u(t,x))$ (see \cite[Theorem 4.4]{Nualart:09} for a similar result which is valid for all $t$ but is not optimal). Compared to Theorem \ref{maint}, one the one hand, we point out that our result is valid for a general $\sigma$, arbitrary time $T>0$ and our estimates look somehow more Gaussian. On the other hand, the general method that we present in Section \ref{subsec:arturo} requires the underlying random variable to be smooth in the Malliavin sense, and this forces to consider a smooth coefficient $b$. We also remark that, even though the results of \cite{Nualart:10} are also valid for a more general class of SPDEs with additive noise (such as the stochastic wave equation in space dimension $d\in \{1,2,3\}$), Nourdin and Viens' method does not seem to be suitable for multiplicative noise settings. 

A second recent method for deriving Gaussian-type lower estimates for multidimensional Wiener functionals has been obtained 
by Malliavin and E.Nualart in \cite{Malliavin} (see \cite{Eulalia} for the one-dimensional counterpart). 
This technique is based on an exponential moment condition on the divergence of a covering vector field associated to the underlying Wiener functional, and has been applied in \cite{Eulalia} to a one-dimensional diffusion.

Last, but not least, in the recent paper \cite{Bally:10}, Bally and Caramellino develop another method to obtain lower bounds for multidimensional Wiener functionals based on the Riesz transform.

As we have already mentioned before, in the present paper we will apply the methodology of Kohatsu-Higa \cite{Kohatsu:03}.
For this, first we will need to extend the general result \cite[Theorem 5]{Kohatsu:03} on Gaussian lower bounds for {\it{uniformly elliptic random vectors}} from the space-time white noise framework to the case of functionals of our Gaussian spatially homogeneous noise (see Theorem \ref{t1}). This will be done in Section \ref{sec:general}, after having precisely described the Gaussian setting which we will work in.
The extension to such a general case turns out to be quite straightforward, since it essentially requires to replace norms in $L^2([0,T]\times A)$, with $A\subseteq \red$, 
by those in $L^2([0,T]; \hac)$, where $\hac$ is the Hilbert space that can be naturally associated to the spatial structure of the noise (see Section \ref{sec:malliavin} below).

In Section \ref{sec:heat}, we will recall the main results on differentiability in the Malliavin sense and existence and smoothness of the density applied to our stochastic 
heat equation (\ref{eq:11}). Moreover, we will prove a technical and useful result which provides a uniform estimate for the conditional norm of the iterated Malliavin derivative of the solution on a small time interval. 

Section \ref{sec:lower-bound} is devoted to apply the general result Theorem \ref{t1} of Section \ref{subsec:arturo} to the stochastic heat equation (\ref{eq:11}), to end up with the 
lower bound in Theorem \ref{maint}. That is, one needs to show that the solution $u(t,x)$ defines a uniformly elliptic random variable in the sense of Definition \ref{def}. 
Although the proof's structure is similar to that of \cite[Theorem 10]{Kohatsu:03}, the analysis in our setting becomes much more involved because of the spatial covariance 
structure of the underlying Wiener noise. As mentioned before, the upper bound in Theorem \ref{maint} will be proved in Section \ref{sec:upper-bound}.

Eventually, we have also added an appendix where, first, we recall some facts concerning Hilbert-space-valued stochastic and 
pathwise integrals and their conditional moment estimates, and, secondly, we state and prove a technical result which has been applied in Section \ref{sec:lower-bound}.

\medskip

As usual, we shall denote by $c, C$ any positive constants whose dependence will be clear from the context and their values may change from one line to another.

\section{General theory on lower bounds for densities}
\label{sec:general}

This section is devoted to extend Kohatsu-Higa's result \cite[Theorem 5]{Kohatsu:03} on lower bounds for the density of a uniformly elliptic random vector to a more general Gaussian space, namely the one determined by a Gaussian random noise on $[0,T]\times \red$ which is white in time and has a non-trivial homogeneous structure in space. For this, first we will rigorously introduce the Gaussian noise and the Malliavin calculus framework associated to it and needed in the sequel.

\subsection{Gaussian context and Malliavin calculus}
\label{sec:malliavin}

Our spatially homogeneous Gaussian noise is described as follows. On a complete probability space $(\Omega, \mathcal{F}, \P)$,
let $W=\{W(\varphi), \varphi \in \mathcal{C}^{\infty}_0(\re_+ \times \R^{d})\}$ be a zero mean Gaussian family
of random variables indexed by $\mathcal{C}^\infty$ functions with compact support with covariance functional given by
\begin{equation}
\E \left[ W(\varphi) W(\psi) \right]= \int_0^\infty dt \int_{\R^d} \Lam(dx) \left( \varphi(t,\star)*\tilde \psi(t,\star)\right)(x), \quad \varphi, \psi\in \cinf_0(\re_+\times \red).
\label{eq:00}
\end{equation}
Here, $\Lam$ denotes a non-negative and non-negative definite tempered measure on $\R^d$, $*$ stands for the convolution product, the symbol $\star$ denotes the spatial variable and $\tilde \psi(t,x):=\psi(t,-x)$. For such a Gaussian process to exist, it is necessary and sufficient that the covariance functional is non-negative definite and this is equivalent to the fact that $\Lam$ is the Fourier transform of a non-negative tempered measure $\mu$ on $\R^d$ (see \cite[Chap. VII, Th\'eor\`eme XVII]{Schwartz}). The measure $\mu$ is usually called the spectral measure of the noise $W$. By definition of the Fourier transform of tempered distributions, $\Lam =\tf \mu$ means that, for all $\phi$ belonging to the space $\cs(\red)$ of rapidly decreasing $\cinf$ functions,
\[
\int_{\red} \phi(x)\Lam(dx)=\int_{\red} \tf \phi(\xi) \mu(d\xi).
\]
Moreover, for some integer $m \geq 1$ it holds that
$$
\int_{\R^d} \frac{\mu(d\xi)}{(1+\Vert \xi \Vert^2)^m} < +\infty.
$$
Elementary properties of the convolution and Fourier transform show that
covariance (\ref{eq:00}) can be written in terms of the measure $\mu$, as follows:
\begin{equation*}
\E \left[ W(\varphi) W(\psi) \right]= \int_0^\infty \int_{\R^d}
\mathcal{F} \varphi(t)(\xi) \overline{\mathcal{F} \psi(t) (\xi)} \mu(d\xi) dt.
\end{equation*}
In particular, we obtain that
\begin{equation*}
\E \left[W(\varphi)^2\right] = \int_0^\infty \int_{\R^d}
|\mathcal{F} \varphi(t)(\xi)|^2 \mu(d\xi) dt.
\end{equation*}

\begin{exa}
Assume that the measure $\Lam$ is absolutely continuous with respect to Lebesgue measure on $\red$ with density $f$. Then, the covariance functional \textnormal{(\ref{eq:00})} reads
\[
\int_0^\infty dt \int_{\red} dx \int_{\red} dy\,  \ffi(t,x) f(x-y) \psi(t,y),
\]
which clearly exhibits the spatially homogeneous nature of the noise. The space-time white noise would correspond to the case where $f$ is the Dirac delta at the origin. 
\end{exa}

We note that the above-defined kind of noise has been widely used as a random perturbation for several classes of SPDEs (see for instance \cite{Millet-Sanz,Dalang:99,Peszat-Zabczyk,Marta}).

\smallskip

At this point, we can describe the Gaussian framework which is naturally associated to our noise $W$. Precisely, 
let $\mathcal{H}$ be the completion of the Schwartz space $\mathcal{S}(\R^{d})$ endowed with the semi-inner product
\begin{equation*}
\langle \phi_1, \phi_2 \rangle_{\mathcal{H}}= \int_{\R^d}  (\phi_1 * \tilde{\phi_2})(x) \, \Lam(dx) =
\int_{\R^d} \mathcal{F} \phi_1(\xi) \overline{\mathcal{F} \phi_2 (\xi)} \, \mu(d\xi), \quad \phi_1, \phi_2 \in \mathcal{S}(\R^{d}).
\end{equation*}
Notice that the Hilbert space $\mathcal{H}$ may contain distributions (see \cite[Example 6]{Dalang:99}). Fix $T>0$ and define  $\mathcal{H}_T=L^2([0,T]; \mathcal{H})$.
Then, the family $W$ can be extended to $\mathcal{H}_T$, so that we end up with a family of centered Gaussian random variables, still denoted by $W=\{W(g), g \in \mathcal{H}_T\}$, satisfying that
$\E[W(g_1)W(g_2)]=\langle g_1,g_2 \rangle_{\mathcal{H}_T}$, for all $g_1,g_2 \in \mathcal{H}_T$ (see for instance \cite[Lemma 2.4]{Dalang-Quer} and the explanation thereafter). 

The family $W$ defines as isonormal Gaussian process on the Hilbert space $\hact$ and we shall use the differential Malliavin calculus based on it (see, for instance, \cite{Nualart:06, Marta}). We denote the Malliavin derivative by
$D$, which is a closed and unbounded operator defined in $L^2(\Om)$ and taking values in  $L^2(\Omega;
\mathcal{H}_T)$, whose domain is denoted by $\mathbb{D}^{1,2}$. More general, for any $m \geq 1$, the domain of the iterated
derivative $D^m$ in $L^p(\Om)$  is denoted by $\mathbb{D}^{m,p}$, for any $p \geq 2$, and we recall that $D^m$ takes values in $L^p(\Omega; \mathcal{H}_T^{\otimes m})$. As usual, we set
$\mathbb{D}^{\infty}= \cap_{p \geq 1} \cap_{m \geq 1}
\mathbb{D}^{m,p}$. The space $\mathbb{D}^{m,p}$ can also be seen as the completion of the set of smooth functionals with respect to the semi-norm
\[
\|F\|_{m,p}=\Big\{\E \left[\vert F \vert^p\right]+\sum_{j=1}^m \E\left[\Vert D^j F \Vert^p_{\hact^{\otimes j}}\right]\Big\}^{\frac 1p}.
\]
For any differentiable random
variable $F$ and any $r=(r_1,...,r_m) \in [0,T]^m$, $D^m F(r)$ is an
element of $\mathcal{H}^{\otimes m}$ which will be denoted by
$D^m_r F$.

We define the Malliavin matrix of a $k$-dimensional random vector $F \in
(\mathbb{D}^{1,2})^{k}$ by $\gamma_F=(\langle DF_i, D F_j
\rangle_{\mathcal{H}_T})_{1\leq i,j\leq k}$.
We will say that a $k$-dimensional random vector $F$ is smooth if each of its components belongs to $\D^{\infty}$, and we will say that a smooth random vector $F$ is non-degenerate if
$(\text{det } \gamma_F)^{-1} \in \cap_{p \geq 1} L^p(\Omega)$.
It is well-known that a non-degenerate random vector has a $\mathcal{C}^{\infty}$
density (cf. \cite[Theorem 2.1.4]{Nualart:06}).

Let $(\mathcal{F}_t)_{t \geq 0}$ denote the $\sigma$-field generated
by the random variables $\{W_s(h), h \in \mathcal{H}, 0 \leq s
\leq t\}$ and the $\P$-null sets, where $W_t(h):=W(1_{[0,t]}h)$,
for any $t \geq 0$, $h \in \mathcal{H}$. Notice that this family defines a standard cylindrical Wiener process on the Hilbert space $\hac$. We define the predictable
$\sigma$-field as the $\sigma$-field in $\Omega \times [0,T]$
generated by the sets $\{ (s,t] \times A, 0\leq s<t \leq T, A \in
\mathcal{F}_s \}$.

As in \cite[Section 2]{Kohatsu:03}, one can define the conditional versions of the above Malliavin norms and spaces
(see also \cite{Moret:01, Dalang:04}). For all $0\leq s <t \leq T$, we set $\mathcal{H}_{s,t}=L^2([s,t]; \mathcal{H})$. For any integer $m \geq 0$ and $p>1$, we define the seminorm:
\begin{equation*}
\Vert F \Vert_{m,p}^{s,t}=\Big\{\E_s[\vert F \vert^p]+\sum_{j=1}^m \E_s\left[\Vert D^j F \Vert^p_{\mathcal{H}_{s,t}^{\otimes j}}\right]\Big\}^{\frac 1p},
\end{equation*}
where $\E_s[\cdot]=\E[\cdot \vert \mathcal{F}_s ]$. We will also write $\P_s\{\cdot\}=\P\{\cdot \vert \mathcal{F}_s \}$. Completing the space of smooth functionals with respect to this seminorm we obtain the space $\D^{m,p}_{s,t}$. We write $L^p_{s,t}(\Omega)$ for $\D^{0,p}_{s,t}$.
We say that $F \in \overline{\D}^{m,p}_{s,t}$ if $F \in \D^{m,p}_{s,t}$ and
$\Vert F \Vert_{m,p}^{s,t} \in \cap_{q \geq 1} L^q(\Omega)$. If we let now $F$ be a $k$-dimensional random vector, we define its associated conditional Malliavin matrix as $\gamma_F^{s,t}=(\langle DF_i, D F_j
\rangle_{\mathcal{H}_{s,t}})_{1\leq i,j\leq k}$.
We say that $F=(F_1,\dots,F_k)$ is conditionally non-degenerate in $[s,t]$ if $F_i \in \overline{\D}^{\infty}_{s,t}=\cap_{m \geq 0, p>1} \overline{\D}^{m,p}_{s,t}$,
for all $i=1,\dots,k$, and $(\text{det } \gamma_F^{s,t})^{-1} \in \cap_{p \geq 1} L_{s,t}^p(\Omega)$.

\subsection{The general result}
\label{subsec:arturo}

In order to state the main result of this section, we need to define what we understand by a uniformly elliptic random vector in our context.
\begin{defi} \label{def}
Let $F$ be a smooth ${\mathcal{F}}_t$-measurable $k$-dimensional random vector.
We say that $F$ is uniformly elliptic if there exists an element $g \in \hact$ such that $\|g(s)\|_\hac >0$ for almost all $s$, and
an $\epsilon>0$ such that,
for any sequence of partitions $\pi_N=\{0=t_0<t_1<\cdots<t_N=T\}$ whose norm is smaller than $\epsilon$
and
$\Vert \pi_N \Vert= \sup \{t_n-t_{n-1}, \, n=1,...,N\}$ converges to zero as $N \rightarrow \infty$,
the following conditions are satisfied:

Define
\begin{equation*}
0< \Delta_{n-1}(g):=\int_{t_{n-1}}^{t_n} \Vert g(s) \Vert^2_{\mathcal{H}} ds
< \infty, \; \; n=1,...,N.
\end{equation*}
\begin{itemize}
\item[\textnormal{\bf (H1)}] (Approximation property) There exists a sequence of smooth random vectors
$F_0,F_1$,...,$F_N$ such that $F_N=F$, $F_n$ is ${\mathcal{F}}_{t_n}$-measurable and, for any $n=1,\dots, N$, $F_n$ is  
conditionally non-degenerate in $[t_{n-1}, t_n]$. Moreover, for some $\gamma>0$ and for each $F_n$ and
$K \in \mathbb{N}$, there exists a random vector 
$\overline{F}_n = \overline{F}_n^K$ of the form
\begin{equation*}
\overline{F}_n = \Delta_{n-1} (g)^{(K+1) \gamma} Z_n+ F_{n-1}+ I^n(h) + G_n^K,
\end{equation*}
where the various terms in this formula and the relationship with $F_n$ are as follows.
$(Z_n, \, n=1,\dots,N)$ is an i.i.d. sequence of $k$-dimensional $N(0,Id)$
random vectors which are independent of the Gaussian family $W$, where $Id$ denotes the $k$-dimensional identity matrix. 
The random vectors $G_n^K$ are ${\mathcal{F}}_{t_n}$-measurable and belong to $\overline{\D}^{\infty}_{t_{n-1}, t_n}$. Eventually, $I^n(h)$ denotes a random vector whose components are of the form
$$(I^n(h))^{\ell}=\int_{t_{n-1}}^{t_n} \int_{\R^d} h^{\ell}(s,y) W(ds,dy), \quad \ell=1,\dots,k,$$
where, for each $\ell$, $h^{\ell}$ is a smooth ${\mathcal{F}}_{t_{n-1}}$-predictable processes with values in $\mathcal{H}_{t_{n-1},t_n}$ and, for any $m \in \mathbb{N}$ and $p \geq 1$, there exists a constant $C$ such that:
\begin{equation*}
\Vert F_n \Vert_{m,p}+ \sup_{\omega \in \Omega} \Vert h^{\ell} \Vert_{\mathcal{H}_{t_{n-1},t_n}}
(\omega) \leq C,
\end{equation*}
for any $\ell=1,\dots,k$. 

Furthermore, the following four conditions are
satisfied for the approximating sequence $\overline{F}_n$, for all $n=1,\dots,N$:
\item[\textnormal{\bf (H2a)}] For any $m \in \mathbb{N}$ and $p \geq 1$:
$$\Vert F_n-\overline{F}_n \Vert_{m,p}^{t_{n-1},t_n}
\leq C\, \Delta_{n-1}(g)^{(K+1)\gamma}\qquad \text{a.s.}$$
\item[\textnormal{\bf (H2b)}] For any $p>1$:
\begin{equation*}
\{\E_{t_{n-1}}[ \textnormal{det} (\gamma^{t_{n-1},t_n}_{F_n})^{-p} ]\}^{1/p}
\leq C\, \Delta_{n-1} (g)^{-k} \qquad \text{a.s.}
\end{equation*}
\item[\textnormal{\bf (H2c)}] Let $A=(a_{\ell,q})$ be the $k \times k$ matrix defined by
\[
a_{\ell,q}= \Delta_{n-1}(g)^{-1} \int_{t_{n-1}}^{t_n}
\langle h^{\ell}(s), h^q(s) \rangle_{\mathcal{H}} ds.
\]
There exist strictly positive constants $C_1$ and $C_2$ such that, for all $\zeta \in \mathbb{R}^k$,
\begin{equation*}
C_1 \zeta^{T} \zeta \geq \zeta^{T} A \zeta \geq C_2 \zeta^{T} \zeta, \quad \text{a.s.}
\end{equation*}
\item[\textnormal{\bf (H2d)}] There exists $\delta >0$ such that, for all $m\in \mathbb{N}$ and $p\geq 1$:
\begin{equation*}
\Vert G_n^K \Vert^{t_{n-1},t_n}_{m, p}\leq C\, \Delta_{n-1} (g)^{1/2+\delta} \qquad \text{a.s.}
\end{equation*}
\end{itemize}
\end{defi}
\vskip 12pt

In hypothesis {\bf{(H1)}} above, the Gaussian random variable $Z_n$ is indeed defined as follows. One enlarges the underlying probability space in order to include the increments of another cylindrical Wiener process $\{\overline{W}_t(h),\, t\geq 0, h\in \hac\}$ independent of $W$. Then, $Z_n=\overline{W}_{n+1}(h)-\overline{W}_n(h)$, where $h$ is any element of a complete orthonormal system on $\hac$. In this case, the expanded filtration is given by $\overline{{\mathcal{F}}}_{t}={\mathcal{F}}_{t} \vee \sigma\{ \overline{W}_s(h), \, s \leq n+1, t_n \leq t, h\in \hac\}$. All the norms and expectations considered so far have to be understood in the enlarged probability space. We remark that, as it has been explained in \cite[p. 436]{Kohatsu:03} for the space-time white noise case, 
the introduction of the term $\Delta_{n-1} (g)^{(K+1) \gamma} Z_n$ is necessary in order to ensure that $\overline{F}_n$ has a density.

On the other hand, note that condition {\bf (H2c)} is the ingredient that most directly reflects the {\it{uniformly elliptic}}
condition for a random vector on the Wiener space. 

\medskip

The next theorem establishes a Gaussian lower bound for the probability density of
a uniformly elliptic random vector. Its proof follows exactly the same steps as that of Theorem 5 in \cite{Kohatsu:03}, where the same result has been proved in a Gaussian setting associated to the Hilbert space $L^2([0,T]\times A)$, where $A\subseteq \red$ (that is, the space-time white noise). In our case, one essentially needs to replace the norms in the latter space by those in $\hact=L^2([0,T];\hac)$, in such a way that the proof follows with minimal changes. Nevertheless, for the sake of completeness, we believe that it is worth reminding the main steps of the proof of this theorem, for the reader will have a much clearer understanding of the long list of assumptions stated above.

\vskip 12pt
\begin{thm}\label{t1}
Let $F$ be a $k$-dimensional uniformly elliptic random vector and denote by $p_F$ its probability density. Then,
there exists a constant $M>0$ that depends on all the constants of Definition \ref{def}
such that:
\begin{equation*}
p_{F}(y) \geq M \Vert g \Vert_{\mathcal{H}_t}^{-k/2}
\exp \biggl(-\frac{\Vert y-F_0 \Vert^2}{M\Vert g \Vert^2_{\mathcal{H}_t}} \biggr),
\; \; \text{for all} \; y \in \mathbb{R}^k,
\end{equation*}
where $F_0$ is the first element in the sequence of hypothesis {\bf{(H1)}}.
\end{thm}

\noindent {\it{Sketch of the proof}}. The proof is divided in four steps.

\noindent {\it Step 1.} First, as we have already mentioned, the fact that in the definition of $\overline{F}_n$ we include the term involving the random variable $Z_n$ let us prove that $\overline{F}_n$ is conditionally non-degenerate in $[t_{n-1},t_n]$.  Therefore, it has a smooth conditional density with respect to
${\mathcal{F}}_{t_{n-1}}$. 

Next, one proves that assumptions {\bf (H2a)}-{\bf (H2d)} imply
the following local lower bound for this conditional
density: there exist constants $M, c$ and $\eta$ such that,
if $\Delta_{n-1} (g)\leq \eta$ and $y \in A_n=\{ y \in \R^k: \Vert y-F_{n-1} \Vert \leq c \, \Delta_{n-1} (g)^{1/2}\}$, then
\begin{equation} \label{aux1}
\E_{t_{n-1}} [\delta_y(\overline{F}_n)] \geq \frac{1}{M \Delta_{n-1} (g)^{k/2}} \; \; \text{a.s.},
\end{equation}
where $\delta_y$ denotes the Dirac delta function. 
This is proved using the expression for the density
that follows from the integration-by-parts formula of the Malliavin calculus.
After normalizing this density (i.e. dividing by $\Delta_{n-1} (g)^{1/2}$),
the Taylor expansion of the delta function around the non-degenerate random vector
$\Delta_{n-1} (g)^{-1/2} I^n(h)$ is considered. Condition {\bf (H2d)} yields that the terms
of higher order, that is, those that are concentrated in the smooth random vector $G^K_n$, are
of order $\Delta_{n-1} (g)^{\delta}$, so that they are negligible
with respect to the first term $I^n(h)$. Using the fact that {\bf (H2c)}
is equivalent to saying that the density of
the term $I^n(h)$ corresponds to a non-degenerate Gaussian random variable,
the lower bound (\ref{aux1}) for the conditional density of $\overline{F}_n$ is obtained.
\vskip 12pt
\noindent {\it Step 2.} The second step consists in proving that, if $\overline{F}_n(\rho):=\rho F_n+(1-\rho)\overline{F_n}$, $\rho\in [0,1]$, then for any $p>1$ there exists a constant $C$ such that:
\begin{equation} \label{aux2}
\sup_{\rho \in [0,1]} \{\E_{t_{n-1}}[ \textnormal{det} (\gamma^{t_{n-1},t_n}_{\overline{F}_n(\rho)})^{-p} ]\}^{1/p}
\leq C\, \Delta_{n-1} (g)^{-k} \qquad \text{a.s.}
\end{equation}
The proof of this inequality is very technical (see \cite[Proposition 12]{Kohatsu:03}), but (\ref{aux2}) is an important ingredient, together with (\ref{aux1}), to end up with a lower bound for the conditional
density of $F_n$ with respect to ${\mathcal{F}}_{ t_{n-1}}$. This is explained in the next step.

\vskip 12pt
\noindent {\it Step 3.} Next, one shows that {\bf (H2a)}, (\ref{aux1}) and (\ref{aux2}) imply that 
there exist positive constants $c, M, \alpha>1, \eta$, and random variables $C_n \in \overline{{\mathcal{F}}}_{t_n}$,
$n=0,...,N-1$, satisfying that $\sup_{n=0,...,N} \E(\vert C_n \vert) \leq M$, such that if $\Delta_{n-1} (g)\leq \eta$ and $y \in A_n$, then
\begin{equation} \label{aux3}
\E_{t_{n-1}} [\delta_y(F_n)] \geq \frac{1}{M \Delta_{n-1} (g)^{k/2}}-C_{n-1}(\omega)  \Delta_{n-1} (g)^{\alpha}, 
\end{equation}
for almost all $w \in \Omega$. 
This lower bound is proved by writing 
\begin{equation*}
\E_{t_{n-1}} [\delta_y(F_n)] \geq \frac{1}{M \Delta_{n-1} (g)^{k/2}}+\E_{t_{n-1}} [\delta_y(F_n)-\delta_y(\overline{F}_n)], 
\end{equation*}
and then finding an upper bound of the second term on the right-hand side using {\bf (H2a)} and (\ref{aux2}).

\vskip 12pt
\noindent {\it Step 4.} Finally, one concludes the desired lower bound of Theorem \ref{t1} using (\ref{aux3}) as in \cite[Theorem 2]{Kohatsu:03}. The main idea is to use Fubini's theorem to write
$$
\E[\delta_y(F)]=\int_{\R^k} \cdots \int_{\R^k} \E[\delta_y(F) \delta_{y_{N-1}}(F_{N-1}) \cdots  \delta_{y_1}(F_1)]
d y_1 \cdots dy_{N-1}.
$$
Then one iteratively applies the lower bound (\ref{aux3}) using conditional densities, and finally use a localization procedure to conclude the proof. 
\hfill $\square$

\section{The stochastic heat equation}
\label{sec:heat}

In this section, we will recall some known facts about the stochastic heat equation on $\red$ which will be needed in the sequel.
We will also prove an estimate involving the iterated Malliavin derivative of the solution which, as far as we know, does not seem to exist in the literature (see Lemma \ref{lema4} below).

\medskip

We remind that the mild solution to the stochastic heat equation (\ref{equa1}) is given by the $\tf_t$-adapted process $\{u(t,x),\, (t,x)\in [0,T]\times \red\}$ that satisfies:
\begin{equation} \begin{split}
u(t,x)= (\Gam(t) \ast u_0)(x)&+ \int_0^t \int_{\red} \Gam(t-s,x-y)\sig(u(s,y)) W(ds,dy)  \\
& + \int_0^t \int_{\red} \Gam(t-s,x-y) b(u(s,y)) \, dy ds,
\label{eq:011}
\end{split}
\end{equation}
where $\Gam(t)$ is the Gaussian kernel with variance $2t$ and the following condition is fulfilled:
\[
\Phi(T)=\int_0^T \int_{\red} |\tf \Gam(t)(\xi)|^2\, \mu(d\xi) dt <\infty.
\]
As mentioned in the Introduction, this is equivalent to say that $\int_{\red} (1+\Vert \xi \Vert^2)^{-1} \mu(d\xi) <\infty$.  
Moreover, similar to \cite[Lemma 3.1]{Marquez}, one easily proves that, for all $0\leq \tau_1 <\tau_2\leq T$:
\beq
C (\tau_2-\tau_1) \leq \int_{\tau_1}^{\tau_2} \int_{\R^d} \vert \mathcal{F}\Gamma(t) (\xi) \vert^2 \mu(d \xi) dt,
\label{eq:112}
\eeq
for some positive constant $C$ depending on $T$.

\medskip

In order to apply the techniques of the Malliavin calculus to the solution of (\ref{eq:11}), let us consider the Gaussian context described in Section \ref{sec:malliavin}. That is, let $\{W(h),\; h\in \hact\}$ be the isonormal Gaussian process on the Hilbert space $\hact=L^2([0,T];\hac)$ defined therein.
Then, the following result is a direct consequence of \cite[Proposition 2.4]{Marquez}, \cite[Theorem 1]{Sanz:04} and \cite[Proposition 6.1]{Nualart:07}.
For the statement, we will use the following notation: for any $m\in \mathbb{N}$, set $\bar s:=(s_1,\dots,s_m)\in [0,T]^m$, $\bar z:=(z_1,\dots,z_m)\in (\red)^m$,
$\bar s(j):=(s_1,\dots,s_{j-1},s_{j+1},\dots,s_m)$ (resp. $\bar z(j)$), and, for any function $f$ and variable $X$ for which it makes sense, set
\[
\Del^m(f,X):= D^mf(X)-f'(X)D^mX.
\]
Note that $\Del^m(f,X)=0$ for $m=1$ and, if $m>1$, it only involves iterated Malliavin derivatives up to order $m-1$.

\begin{prop}\label{prop:dinf}
Assume that \textnormal{(\ref{eq:111})} is satisfied and $\sigma, b \in \mathcal{C}^{\infty}(\R)$ and their derivatives of order greater than or equal to one are bounded.
Then, for all $(t,x)\in [0,T]\times \red$, the random variable $u(t,x)$ belongs to $\D^\infty$. Furthermore, for any $m\in \N$ and $p\geq 1$, the iterated Malliavin derivative $D^mu(t,x)$ satisfies the following equation in $L^p(\Om;\hact^{\otimes m})$:
\begin{equation} \begin{split}
D^mu(t,x)& = Z^m(t,x)  \\
& + \int_0^t\int_{\red} \Gam(t-s,x-y) [\Del^m(\sig,u(s,y))+D^m u(s,y) \sig'(u(s,y))] W(ds,dy)  \\
& + \int_0^t\int_{\red} \Gam(t-s,x-y) [\Del^m(b,u(s,y))+D^m u(s,y) b'(u(s,y))]\, dy ds,
\label{eq:113}
\end{split}
\end{equation}
where $Z^m(t,x)$ is the element of $L^p(\Om;\hact^{\otimes m})$ given by
\[
Z^m(t,x)_{\bar s,\bar z}= \sum_{j=1}^m \Gam(t-s_j,x-z_j) D^{m-1}_{\bar s(j), \bar z(j)} \sig(u(s_j,z_j)).
\]
\end{prop}

We remark that the Hilbert-space-valued stochastic and pathwise integrals in equation (\ref{eq:113}) are understood as it has been described in Section \ref{Hilbert-integrals}.

As far as the existence of a smooth density is concerned, we have the following result (see \cite[Theorem 6.2]{Nualart:07}):
\begin{thm}\label{thm:smooth-density}
Assume that \textnormal{(\ref{eq:111})} is satisfied and $\sigma, b \in \mathcal{C}^{\infty}(\R)$ and their derivatives of order greater than or equal to one are bounded.
Moreover, suppose that $\vert \sigma(z) \vert \geq c >0$, for all $z \in \R$. Then, for every $(t,x)\in (0,T]\times \red$, the law of the random variable $u(t,x)$ has a $\cinf$ density.
\end{thm}

The following technical result, which will be used in the proof of Theorem \ref{maint}, exhibits an almost sure estimate for the conditional moment of the iterated Malliavin derivative of $u$ in a small time interval. 
As will be explained in Remark \ref{rmk:general-gam}, this result is still valid for a slightly more general class of SPDEs, such as the stochastic wave equation in space dimension $d\in \{1,2,3\}$. Nevertheless, for the sake of simplicity, we will focus either the statement and its proof on our stochastic heat equation (\ref{eq:11}).  

\begin{lem} \label{lema4}
Let $0\leq a<e\leq T$, $m\in \N$ and $p\geq 1$. Assume that the coefficients $b,\sig:\re\rightarrow \re$ belong to $\cinf(\re)$ and all their derivatives of order greater than or equal to one are bounded. Then, there exists a positive constant $C$, which is independent of $a$ and $e$, such that, for all $\del\in (0,e-a]$:
\begin{equation*}
\sup_{(\tau,y) \in [e-\del,e]\times \red} \E_a \left( \Vert D^m u(\tau,y) \Vert^{2p}_{\mathcal{H}^{\otimes m}_{e-\del,e}} \right) \leq C
\, \left(\Phi(\del)\right)^{mp}, \quad a.s.,
\end{equation*}
where we remind that $\mathcal{H}^{\otimes m}_{e-\del,e}$ denotes the Hilbert space $L^2([e-\del,e];\hac^{\otimes m})$ and, for all $t\geq 0$,
\[
\Phi(t)= \int_0^t \int_{\red} |\tf \Gam(s)(\xi)|^2\, \mu(d\xi) ds.
\]
\end{lem}

\begin{proof}
We will proceed by induction with respect to $m\in \N$. First, let us observe that the case $m=1$ has been proved in \cite[Lemma 2.5]{Marquez} (see also
\cite[Lemma 5]{Sanz:04}). Suppose now that the statement holds for any $j=1,\dots,m-1$, and let us check its veracity for $j=m$.

Let $e-\del\leq t\leq e$ and $x\in \red$. Then, the conditioned norm of the Malliavin derivative $D^mu(t,x)$ can be decomposed as follows:
\[
E_a \left( \|D^mu(t,x)\|^{2p}_{\hac^{\otimes m}_{e-\del,e}}\right)\leq C(B_1+B_2+B_3), \quad a.s.
\]
with
\[
B_1= E_a \left( \int_{(e-\del,e)^m} \Big\| \sum_{j=1}^m \Gam(t-s_j,x-\star) D^{m-1}_{\bar s(j)} \sig(u(s_j,\star))\Big\|^2_{\hac^{\otimes m}}
d\bar s \right)^p
\]
(here, if we formally denote by $(z_1,\dots,z_m)$ the variables of $\hac^{\otimes m}$, the symbol $\star$ corresponds to $z_j$),
\begin{equation*} \begin{split}
B_2&= E_a \left( \left\| \int_0^t\int_{\red} \Gam(t-s,x-y) [\Del^m(\sig,u(s,y))-D^mu(s,y) \sig'(u(s,y))]W(ds,dy)\right\|^{2p}_{\hac^{\otimes m}_{e-\del,e}}\right), \\
B_3&= E_a \left( \left\| \int_0^t\int_{\red} \Gam(t-s,x-y) [\Del^m(b,u(s,y))-D^mu(s,y) b'(u(s,y))] \, dy ds \right\|^{2p}_{\hac^{\otimes m}_{e-\del,e}}\right).
\end{split}
\end{equation*}
Let us start with the study of the term $B_1$. First, note that we must have that $e-\del\leq s_j\leq t$, thus
\[
B_1 \leq C \sum_{j=1}^m E_a \left( \int_{e-\del}^t ds_j \int_{(e-\del,e)^{m-1}} d\bar s(j) \Big\| \Gam(t-s_j,x-\star) D^{m-1}_{\bar s(j)}
\sig(u(s_j,\star))\Big\|^2_{\hac^{\otimes m}} \right)^p.
\]
At this point, we can proceed as in the proof of \cite[Lemma 2]{Sanz:04} (see p. 173 therein), so that we can infer that
\[
B_1 \leq C \left(\int_{e-\del}^t J(t-r)\, dr\right)^{p-1} \int_{e-\del}^t \sup_{y\in\red} E_a \left( \|D^{m-1} \sig(u(r,y))\|^{2p}_{\hac^{\otimes (m-1)}_{e-\del,e}} \right) J(t-r)\, dr,
\]
where we have used the notation 
\[
J(r)=\int_{\red} |\tf \Gam(r)(\xi)|^2\, \mu(d\xi).
\]
Precisely, we have used the fact that $\Gam$ is a smooth function, and then applied H\"older's and Cauchy-Schwarz inequalities. Hence, we have that
\beq
B_1 \leq C \left(\Phi(\del)\right)^p \sup_{(r,y) \in [e-\del,e]\times \red} E_a \left( \|D^{m-1} \sig(u(r,y))\|^{2p}_{\hac^{\otimes (m-1)}_{e-\del,e}} \right)
\label{eq:52}
\eeq
In order to bound the above supremum, one applies the Leibniz rule for the iterated Malliavin derivative, the smoothness assumptions on $\sig$, H\"older's inequality and the induction hypothesis, altogether yielding
\[
\sup_{(r,y) \in [e-\del,e]\times \red} E_a \left( \|D^{m-1} \sig(u(r,y))\|^{2p}_{\hac^{\otimes (m-1)}_{e-\del,e}} \right) \leq C (\Phi(\del))^{(m-1)p}, \quad a.s.
\]
Plugging this bound in (\ref{eq:52}), we end up with
\beq
B_1 \leq C (\Phi(\del))^{mp}, \quad a.s.
\label{eq:53}
\eeq

Next, we will deal with the term $B_2$, which will be essentially bounded by means of Lemma \ref{lema2}, as follows:
\begin{align*}
B_2 & \leq C \left(\int_{e-\del}^t J(t-r)\, dr\right)^{p-1} \nonumber \\
& \quad \times \int_{e-\del}^t \left[ \sup_{y\in \red} E_a \left( \|\Del^m(\sig,u(s,y))\|^{2p}_{\hac^{\otimes m}_{e-\del,e}} \right) +
\sup_{y\in \red} E_a \left( \|D^m u(s,y)\|^{2p}_{\hac^{\otimes m}_{e-\del,e}} \right)\right] J(t-s) ds.
\end{align*}
Owing again to the the Leibniz rule for the Malliavin derivative and noting that $\Del^m$ only involves Malliavin derivatives up to order $m-1$, one makes use of the induction hypothesis to infer that
\[
\sup_{y\in \red} E_a \left( \|\Del^m(\sig,u(s,y))\|^{2p}_{\hac^{\otimes m}_{e-\del,e}} \right) \leq C (\Phi(\del))^{mp}, \quad a.s.
\]
Hence,
\begin{equation} \begin{split}
B_2 & \leq C (\Phi(T))^{p-1} \int_{e-\del}^t \left[ (\Phi(\del))^{mp} +
\sup_{y\in \red} E_a \left( \|D^m u(s,y)\|^{2p}_{\hac^{\otimes m}_{e-\del,e}} \right)\right] J(t-s) ds \\
& \leq C_1 \int_{e-\del}^t \left[ (\Phi(\del))^{mp} +
\sup_{(\tau,y)\in [e-\del,s]\times \red}
E_a \left( \|D^m u(\tau,y)\|^{2p}_{\hac^{\otimes m}_{e-\del,e}} \right)\right] J(t-s) ds,
\label{eq:54}
\end{split}
\end{equation}
almost surely, where $C_1$ denotes some positive constant.

Eventually, using similar arguments, we can show that the term $B_3$ is bounded above by:
\begin{equation} \begin{split}
&C \int_{e-\del}^t\int_{\red} \Gam(t-s,x-y) \left[ E_a \left( \|\Del^m(b,u(s,y))\|^{2p}_{\hac^{\otimes m}_{e-\del,e}} \right) + E_a \left( \|D^m u(s,y)\|^{2p}_{\hac^{\otimes m}_{e-\del,e}} \right)\right] dy ds \\
& \leq C \int_{e-\del}^t \left[ (\Phi(\del))^{mp} +
\sup_{(\tau,y)\in [e-\del,s]\times \red}
E_a \left( \|D^m u(\tau,y)\|^{2p}_{\hac^{\otimes m}_{e-\del,e}} \right)\right] ds, \quad a.s.
\label{eq:55}
\end{split}
\end{equation}
Here, we have also used that $\int_{\red} \Gam(s,y)\, dy$ is uniformly bounded with respect to $s$.

Set
\[
F(t):=\sup_{(s,y)\in [e-\del,t]\times \red} E_a \left( \|D^m u(s,y)\|^{2p}_{\hac^{\otimes n}_{e-\del,e}} \right), \quad t\in [e-\del,e].
\]
Then, (\ref{eq:53})-(\ref{eq:55}) imply that
\[
F(t)\leq C_2 (\Phi(\del))^{mp} + C_1 \int_{e-\del}^t [(\Phi(\del))^{mp} + F(s)] (J(t-s)+1) ds, \quad a.s.,
\]
where $C_1$ and $C_2$ are some positive constants. We conclude the proof by applying Gronwall's lemma \cite[Lemma 15]{Dalang:99}.
\end{proof}

\begin{remark}\label{rmk:general-gam}
Lemma \ref{lema4} still remains valid for a more general class of SPDEs, namely for those that have been considered in the paper \cite{Dalang:99} (see also \cite{Nualart:07}). In these references, an SPDE driven by a linear second-order partial differential operator has been considered, where one assumes that the corresponding fundamental solution $\Gam$ satisfies the following: for all $s$, $\Gam(s)$ is a non-negative measure which defines a distribution with rapid decrease such that condition (\ref{hyp1}) is fulfilled and
\[
\sup_{0\leq s\leq T}\Gam(s,\red)<+\infty.
\]
As explained in \cite[Section 3]{Dalang:99}, together with the stochastic heat equation, the stochastic wave equation in space dimension $d\in \{1,2,3\}$ is another example of such a type of equations.
Finally, we point out that the proof of Lemma \ref{lema4} in such a general setting would require a smoothing procedure of $\Gam$ in terms of an approximation of the identity, which makes the proof slightly 
longer and more technical; this argument has been used for instance in \cite[Lemma 5]{Sanz:04}.
\end{remark}

\section{Proof of the lower bound}
\label{sec:lower-bound}

In this section, we prove the lower bound in the statement of Theorem \ref{maint}. For this,
we are going to show that $u(t,x)$ is a uniformly elliptic random variable in the sense of Definition \ref{def}.
Then, an application of Theorem \ref{t1} will give us the desired lower bound.

\medskip

To begin with, we fix $(t,x) \in (0,T] \times \R^d$, we consider a partition $0=t_0<t_1<\cdots<t_N=t$ whose norm converges to zero, and define:
\begin{equation*} \begin{split}
F_n= (\Gam(t)\ast u_0)(x) &+\int_0^{t_n} \int_{\R^d} \Gamma(t-s, x-y) \sigma(u(s,y)) W(ds, dy) \\
&+\int_0^{t_n} \int_{\R^d} \Gamma(t-s, x-y) b(u(s,y))  dsdy.
\end{split}
\end{equation*}
Clearly, $F_n$ is ${\mathcal{F}}_{t_n}$-measurable, for all $n=0,...,N$ and note that $F_0= (\Gam(t)\ast u_0)(x) $.
Moreover, $F_n$ belongs to $\D^\infty$ and, for all $m\in \mathbb{N}$ and $p\geq 1$, the norm $\Vert F_n \Vert_{m,p}$ can be uniformly bounded with respect to $(t,x)\in (0,T]\times \red$ 
(see \cite[Proposition 2.4]{Marquez}, \cite[Theorem 1]{Sanz:04}, and also \cite[Proposition 6.1]{Nualart:07}). 

The local variance of the random variable $u(t,x)$ will be measured through the function $g(s):=\Gamma(t-s)$. Then, observe that we have:
$$
\Delta_{n-1} (g)=\int_{t_{n-1}}^{t_n} \int_{\R^d} \vert \mathcal{F}\Gamma(t-s) (\xi) \vert^2 \mu(d
\xi) ds,
$$
and this quantity is clearly positive for all $n$ (see (\ref{eq:111})).
 
We next prove that condition {\bf (H2b)} is satisfied, which in particular implies that $F_n$ is conditionally non-degenerate in $[t_{n-1}, t_n]$. Recall that we are assuming that the coefficients $b$ and $\sig$ belong to $\mathcal{C}^\infty_b(\re)$.
\begin{prop}\label{prop1}
For any $p>0$, there exists a constant $C$ such that:
\begin{equation} \label{aux}
\E_{t_{n-1}} \biggl[ \bigg\vert \int_{t_{n-1}}^{t_n} \Vert D_r F_n \Vert^2_{\mathcal{H}} \, dr \bigg\vert^{-p} \biggr]
\leq C\, \Delta_{n-1} (g)^{-p} \qquad \text{a.s.}
\end{equation}
\end{prop}

\begin{proof}
It follows similarly as the proof of Theorem 6.2 in \cite{Nualart:07}.
More precisely, it suffices to show that for any $q>2$, there exists $\ep_0=\epsilon_0(q)>0$ such that, for all $\epsilon \leq \epsilon_0$:
\begin{equation} \label{cc}
\P_{t_{n-1}} \biggl\{\Delta_{n-1}^{-1} (g) \int_{t_{n-1}}^{t_n} \Vert D_r F_n \Vert^2_{\mathcal{H}} \, dr <\epsilon\biggr\} \leq C \epsilon^q \qquad \text{a.s.}
\end{equation}
Indeed, if we set $X=\Delta_{n-1}^{-1} (g) \int_{t_{n-1}}^{t_n} \Vert D_r F_n\Vert^2_{\mathcal{H}} \, dr$, then we have:
\begin{equation*}
\E_{t_{n-1}} [X^{-p}]=\int_0^{\infty} p y^{p-1} \P_{t_{n-1}} \left\{ X< \frac{1}{y}\right\}\, dy \qquad \text{a.s.}
\end{equation*}
Choosing $q$ sufficiently large in (\ref{cc}) (namely $q>p$), we conclude that (\ref{aux}) is fulfilled, and hence the statement of {\bf (H2b)}.

We next prove (\ref{cc}).
For any $\delta\in (0,t_n-t_{n-1})$, we have that
\begin{equation*}
\begin{split}
\int_{t_{n-1}}^{t_n} \Vert D_r F_n \Vert^2_{\mathcal{H}} \, dr\geq \frac{\sig_0^2}{2} I_0 -(2I_1+2I_2),
\end{split}
\end{equation*}
where
\begin{equation*}
\begin{split}
I_0&=\int_{t_n-\delta}^{t_n}
\int_{\R^d} \vert \mathcal{F}\Gamma(t-r) (\xi) \vert^2 \mu(d\xi) dr, \\
I_1&= \int_{t_n-\delta}^{t_n} \bigg\Vert \int_{r}^{t_n}
\int_{\R^d} \Gamma(t-s, x-z) \sigma'(u(s,z)) D_{r}u(s,z) \, W(ds,dz) \bigg\Vert^2_{\mathcal{H}} dr, \\
I_2&= \int_{t_n-\delta}^{t_n} \bigg\Vert \int_{r}^{t_n}
\int_{\R^d} \Gamma(t-s, x-z) b'(u(s,z)) D_{r}u(s,z) \, dzds \bigg\Vert^2_{\mathcal{H}} dr,
\end{split}
\end{equation*}
and where we have denoted by $\sig_0$ the constant such that $|\sig(v)|\geq \sig_0>0$, for all $v\in \re$.

We next bound the $p$th moments of $I_1$ and $I_2$ for $p >1$.
Owing to H\"older's inequality and Lemmas \ref{lema2} and \ref{lema4}, we get:
\begin{equation*} \begin{split}
\E_{t_{n-1}} [\vert I_1 \vert^{p}] & \leq
\delta^{p-1} \, \E_{t_{n-1}} \biggl[ \int_{t_n-\delta}^{t_n} \bigg\Vert \int_{t_n-\delta}^{t_n}
\int_{\R^d} \Gamma(t-s, x-z) \\
& \qquad \qquad \times \sigma'(u(s,z)) D_{r}u(s,z) \, W(ds,dz) \bigg\Vert^{2p}_{\mathcal{H}} dr \biggr] \\
& = \delta^{p-1} (t_n-t_{n-1})^{p-1} \E_{t_{n-1}} \biggl[ \int_0^{\del} \bigg\Vert \int_{t_n-\delta}^{t_n}
\int_{\R^d} \Gamma(t-s, x-z) \\
& \qquad \qquad \times \sigma'(u(s,z)) D_{t_n-r}u(s,z) \, W(ds,dz) \bigg\Vert^{2p}_{\mathcal{H}} dr \biggr] \\
& \leq
\delta^{p-1} I_0^p \sup_{\stackrel{s \in [0, \del]}{z\in \red} } \E_{t_{n-1}} \biggl[ \Vert D_{t_n-\cdot} u(t_n-s,z)
\Vert^{2p}_{\mathcal{H}_{\del}} \biggr] \\
&\leq \delta^{p-1} \, I_0^p \, \Phi(\del)^p.
\end{split}
\end{equation*}
Similarly, appealing to the (conditional) H\"older's inequality, and Lemmas \ref{lema3} and \ref{lema4},
we have:
\begin{equation*} \begin{split}
\E_{t_{n-1}} [\vert I_2 \vert^{p}] & \leq \delta^{p-1} \,\E_{t_{n-1}} \biggl[ \int_0^{\delta}
\bigg\Vert \int_{t_n-\delta}^{t_n}
\int_{\R^d} \Gamma(t-s, x-z)
b'(u(s,z)) D_{t_n-r}u(s,z) \, dzds \bigg\Vert^{2p}_{\mathcal{H}} dr \biggr] \\
& \leq \delta^{p-1} \left(\int_{t_n-\delta}^{t_n} \int_{\red}\Gam(t-s,z)\,dz ds\right)^{2p}
\sup_{ \stackrel{s \in [0, \del]}{z\in \red} } \E_{t_{n-1}} \biggl[ \Vert D_{t_n-\cdot} u(t_n-s,z) \Vert^{2p}_{\mathcal{H}_{\delta}} \biggr] \\
&\leq \delta^{p-1} \left(\int_{t_n-\delta}^{t_n} \int_{\red} \Gam(t-s,z)\,dz ds\right)^{2p} \, \Phi(\del)^p.
\end{split}
\end{equation*}
Observe that, because $\Gam$ is a Gaussian density,
\[
\bar{I}_0:= \int_{t_n-\delta}^{t_n} \int_{\red} \Gam(t-s,z)\, dz ds \leq C\, \del.
\]
Putting together the bounds obtained so far, we have proved that:
\begin{equation*}
\E_{t_{n-1}} [\vert 2I_1+2I_2 \vert^p] \leq \delta^{p-1} \Phi(\del)^p
\left( I_0^p + \bar{I}_0^p\right) \qquad \text{a.s.}
\end{equation*}
Thus, the (conditional) Chebyschev's inequality yields
\begin{equation}
\begin{split}
&\P_{t_{n-1}} \biggl\{\Delta_{n-1}^{-1} (g) \int_{t_{n-1}}^{t_n} \Vert D_r F_n \Vert^2_{\mathcal{H}} \, dr <\epsilon\biggr\}\\
& \leq C \biggl( \frac{\sig_0^2}{2} \frac{I_0}{\Delta_{n-1}(g)}-\epsilon \biggr)^{-p}
(\Delta_{n-1}(g))^{-p} \, \delta^{p-1} \Phi(\del)^p \left( I_0^p + \bar{I}_0^p\right) \qquad \text{a.s.}
\end{split}
\label{eq:10}
\end{equation}
At this point, we choose $\delta=\delta(\epsilon)$ such that $\frac{I_0}{\Delta_{n-1}(g)}=\frac{4}{\sig_0^2}\epsilon$.
Thus, condition (\ref{eq:112}) implies that $\frac{4}{\sig_0^2}\epsilon\geq \frac{C \del}{\Delta_{n-1}(g)}$, that is
$\delta \leq C \epsilon$. Hence, plugging this bound in (\ref{eq:10}), we obtain:
\begin{align*}
\P_{t_{n-1}} \biggl\{\Delta_{n-1}^{-1} (g) \int_{t_{n-1}}^{t_n} \Vert D_r F_n \Vert^2_{\mathcal{H}} \, dr <\epsilon\biggr\} &
\leq C \, I_0^{-p} \, \ep^{p-1} \left( I_0^p + \bar{I}_0^p\right) \\
& \leq C\, \ep^{p-1}.
\end{align*}
In order to obtain (\ref{cc}), it suffices to choose $p$ sufficiently large such that $p-1 \geq q$. The proof of (\ref{aux}) is now complete.
\end{proof}

\begin{remark}\label{rmk:heat}
In the above proof, we have used the lower bound (\ref{eq:112}). This has prevented us from proving our main result Theorem \ref{maint} for other type of SPDEs, such as the stochastic wave equation (see Remark \ref{rmk:general-gam}). Indeed, for the latter SPDE, we do not have a kind of time homogeneous lower bound of the form (\ref{eq:112}), which has been a key point in order to conclude the proof of Proposition \ref{prop1}. 
\end{remark}

\medskip

In order to define the approximation sequence $\overline{F}_n$, we proceed similarly as in \cite[p. 442]{Kohatsu:03}.
Precisely, we observe first that
\begin{equation*}
\begin{split}
F_n-F_{n-1}&= \int_{t_{n-1}}^{t_n} \int_{\R^d} \Gamma(t-s, x-y) \sigma(u(s,y)) W(ds, dy) \\
&\qquad +\int_{t_{n-1}}^{t_n} \int_{\R^d} \Gamma(t-s, x-y) b(u(s,y))  dsdy.
\end{split}
\end{equation*}
Our aim is to find a Taylor-type expansion of the two terms above. This will be done by applying the mean value theorem to the
functions $\sigma$ and $b$ around the point $u_{n-1}(s,y)$ defined by
\begin{equation*}
\begin{split}
u_{n-1}(s,y)&= \int_{\red} \Gam(s,y-z)u_0(z)dz+ \int_0^{t_{n-1}} \int_{\R^d} \Gamma(s-r, y-z) \sigma(u(r,z)) W(dr, dz) \\
& \qquad +\int_{0}^{t_{n-1}} \int_{\R^d} \Gamma(s-r, y-z) b(u(r,z))  \, dz dr,
\end{split}
\end{equation*}
where $(s,y)\in [t_{n-1},t_n]\times \red$. We clearly have that $u_{n-1}(s,y)$ is ${\mathcal{F}}_{t_{n-1}}$-measurable and  belongs to $\D^{\infty}$.

We next consider the difference:
\begin{equation} \label{difu}
\begin{split}
u(s,y)-u_{n-1}(s,y)&= \int_{t_{n-1}}^s \int_{\R^d} \Gamma(s-r, y-z) \sigma(u(r,z)) W(dr, dz) \\
&\qquad +\int_{t_{n-1}}^s \int_{\R^d}  \Gamma(s-r, y-z) b(u(r,z)) drdz.
\end{split}
\end{equation}
We have the following estimate for the norm in $\mathbb{D}_{t_{n-1},s}^{m,p}$ of the above difference:
\begin{lem} \label{lema8}
For all $m \in \mathbb{N}$ and $p>1$, there exists a constant $C$ such that, for any $s \in[t_{n-1}, t_n]$, we have:
\begin{equation*}
\Vert u(s,y)-u_{n-1}(s,y) \Vert^{t_{n-1},s}_{m,p} \leq C \, (\Phi(s-t_{n-1}))^{1/2}.
\end{equation*}
\end{lem}

\begin{proof}
First, observe that Lemma \ref{lema4} implies that,  for all $j=0,\dots,m$:
\begin{equation*}
\sup_{(\tau,z) \in [t_{n-1},s] \times \R^d}  E_{t_{n-1}} \left( \|D^j u(\tau,z)\|^p_{\hac_{t_{n-1},s}^{\otimes j}} \right) 
\leq C \, \left( \Phi(s-t_{n-1})\right)^{\frac{jp}{2}},\qquad \text{a.s.} 
\end{equation*}
At this point, we can apply Lemma \ref{lema14} in the following situation: $X=u(s,y)-u_{n-1}(s,y)$, $X_0=0$, $f=b$, $g=\sigma$, $v=u$, $I_i(s,y)\equiv1$, $a=t_{n-1}$, $b=s$, $t=s$, $\gamma=0$; we point out that here we are doing an abuse of notation, since $b$ in the present section denotes a different type of object in comparison with the statement of Lemma \ref{lema14} (here $b$ is the drift coefficient while in the lemma it is simply a time parameter). Hence, one ends up with 
\begin{equation*} \begin{split}
\Vert u(s,y)-u_{n-1}(s,y) \Vert^{t_{n-1},s}_{m,p} &\leq C\, \left\{ (s-t_{n-1})^p+\left(\int_{t_{n-1}}^s J(s-r)dr\right)^{\frac{p}{2}} \right\}^{\frac 1p}\\
&\leq C (\Phi(s-t_{n-1}))^{1/2},
\end{split}
\end{equation*}
where the constant $C$ depends on $(m,p,T)$ and we recall that $J(r)=\int_{\red} |\tf \Gam(r)(\xi)|^2 \mu(d\xi)$. The latter inequality has been obtained by applying (\ref{eq:112}). Therefore, we conclude the proof. 
\end{proof}

Let us start with the decomposition of the term $F_n-F_{n-1}$. For this, we set $u(\lambda, s,y):=\lambda u(s,y)+(1-\lambda) u_{n-1}(s,y)$, for $\lambda \in [0,1]$. Then, by the mean value theorem we have:
\begin{equation} \label{aauu}
\begin{split}
&F_n-F_{n-1}= \int_{t_{n-1}}^{t_n} \int_{\R^d} \Gamma(t-s, x-y) \sigma(u_{n-1}(s,y)) W(ds, dy) \\
&\qquad +\int_{t_{n-1}}^{t_n} \int_{\R^d}  \Gamma(t-s, x-y) b(u_{n-1}(s,y)) \, dyds \\
&\qquad +\int_{t_{n-1}}^{t_n} \int_{\R^d} \Gamma(t-s, x-y) \biggl(\int_0^1 \sigma'(u(\lambda, s,y))\, d \lambda\biggr)
(u(s,y)-u_{n-1}(s,y)) W(ds, dy) \\
&\qquad +\int_{t_{n-1}}^{t_n} \int_{\R^d} \Gamma(t-s, x-y) \biggl(\int_0^1 b'(u(\lambda, s,y))\, d \lambda\biggr)
(u(s,y)-u_{n-1}(s,y))  \, dyds.
\end{split}
\end{equation}
As we will make precise below, the first and second terms on the right-hand side of (\ref{aauu}) will be called {\it{processes of order}} $1$ and $2$, respectively, while the third and fourth terms will be called {\it{residues of order}} $1$ and $2$, respectively.

At this point, we need to introduce some notation, namely we are going to define what we understand by {\it{processes and residues of order $k\in \mathbb{N}$}}. The former will be denoted by $J_k$ and the latter by $R_k$. In all the definitions that follow, we assume that $s\leq t$.

For $k=1$, we define:
\begin{equation*}
\begin{split}
J_1(s,t,x)= \int_{t_{n-1}}^{s} \int_{\R^d} \Gamma(t-r, x-y) \sigma(u_{n-1}(r,y)) W(dr, dy).
\end{split}
\end{equation*}
For $k \geq 2$, the process $J_k$ of order $k$ is defined either of the form:
\begin{equation} \label{Jk1}
J_k(s,t,x)=\int_{t_{n-1}}^{s} \int_{\R^d} \Gamma(t-r, x-y) \sigma^{(\ell)}(u_{n-1}(r,y)) \prod_{j=1}^{\ell} J_{m_j}(r,r,y) W(dr, dy),
\end{equation}
where $\ell \leq k-1$ and $m_1+ \cdots +m_{\ell}=k-1$, or
\begin{equation} \label{Jk1b}
\begin{split}
J_k(s,t,x)=\int_{t_{n-1}}^{s} \int_{\R^d} \Gamma(t-r, x-y) b^{(\ell)}(u_{n-1}(r,y)) \prod_{j=1}^{\ell} J_{m_j}(r,r,y) \, dy dr,
\end{split}
\end{equation}
where $\ell \leq k-2$ and $m_1+ \cdots +m_{\ell}=k-2$; in the case $k=2$, the process $J_k$ is defined by
\begin{equation*}
\begin{split}
J_2(s,t,x)=\int_{t_{n-1}}^{s} \int_{\R^d} \Gamma(t-r, x-y) b(u_{n-1}(r,y))  \, dy dr.
\end{split}
\end{equation*}
It is clear that, for any $k\in \mathbb{N}$, the set of processes of order $k$ is finite. Let $\mathcal{A}_k$ be an index set for the family of processes of order $k$, so that if $\alpha \in \mathcal{A}_k$, the corresponding process of order $k$ indexed by $\alpha$ will be denoted by $J_k^\al$.

The residues $R_k$ of order $k$ are defined as follows. For $k=1$, $R_1$ is defined to be either:
\begin{equation} \label{res1a}
\begin{split}
R_1(s,t,x)=\int_{t_{n-1}}^{s} \int_{\R^d} \Gamma(t-r, x-y) \biggl(\int_0^1 \sigma'(u(\lambda, r,y))\, d \lambda\biggr)
(u(r,y)-u_{n-1}(r,y)) W(dr, dy)
\end{split}
\end{equation}
or
\begin{equation} \label{res1b}
\begin{split}
R_1(s,t,x)=\int_{t_{n-1}}^{s} \int_{\R^d} \Gamma(t-r, x-y)  b(u(r,y)) \, dydr .
\end{split}
\end{equation}
For $k=2$, set:
\begin{equation} \label{r2b}
\begin{split}
R_2(s,t,x)=\int_{t_{n-1}}^{s} \int_{\R^d} \Gamma(t-r, x-y)\biggl(\int_0^1 b'(u(\lambda, r,y))\, d \lambda\biggr)
(u(r,y)-u_{n-1}(r,y))  dr dy.
\end{split}
\end{equation}
Eventually, for any $k\geq 3$, the residue of order $k$ can be one of the following four possibilities. First, it can be either:
\begin{equation} \label{Rk1}
\begin{split}
R_k(s,t,x)&=\frac{1}{(k-1)!} \int_{t_{n-1}}^{s} \int_{\R^d} \Gamma(t-r, x-y) \biggl(\int_0^1 (1-\lambda)^{k-1} \sigma^{(k)}(u(\lambda, r,y))\, d \lambda\biggr) \\
& \qquad \times (u(r,y)-u_{n-1}(r,y))^k W(dr, dy)
\end{split}
\end{equation}
or
\begin{equation}  \label{Rk2}
\begin{split}
R_k(s,t,x)=\int_{t_{n-1}}^{s} \int_{\R^d} \Gamma(t-r, x-y) \sigma^{(\ell)}(u_{n-1}(r,y)) \prod_{j=1}^{\ell} I_{m_j}(r,y) W(dr, dy),
\end{split}
\end{equation}
where $I_{m_j}(r,y)$ is either $R_{m_j}(r,r,y)$ or $J_{m_j}(r,r,y)$, but at least there is one $j$ such that
$I_{m_j}(r,y)=R_{m_j}(r,r,y)$. As before $\ell \leq k-1$ and $m_1+ \cdots +m_{\ell}=k-1$. Secondly, $R_k$ can be either:
\begin{equation}  \label{Rk3}
\begin{split}
R_k(s,t,x)&=\frac{1}{(k-2)!} \int_{t_{n-1}}^{s} \int_{\R^d}\Gamma(t-r, x-y) \biggl(\int_0^1 (1-\lambda)^{k-2} b^{(k-1)}(u(\lambda, r,y))\, d \lambda\biggr) \\
& \qquad \times (u(r,y)-u_{n-1}(r,y))^{k-1}  \, dy dr
\end{split}
\end{equation}
or
\begin{equation}  \label{Rk4}
\begin{split}
R_k(s,t,x)=\int_{t_{n-1}}^{s} \int_{\R^d}  \Gamma(t-r, x-y)  b^{(\ell)}(u_{n-1}(r,y)) \prod_{j=1}^{\ell} I_{m_j}(r,y) \, dy dr,
\end{split}
\end{equation}
where $\ell \leq k-2$, $m_1+ \cdots +m_{\ell}=k-2$ and $I_{m_j}$ are as in (\ref{Rk2}). We will denote here by $\mathcal{B}_k$ any index set for the residues of order $k$ and $R_k^\al$ the element corresponding to $\al\in \mathcal{B}_k$.

\medskip

Having all these notation in mind, we have the following decomposition for the difference $F_n-F_{n-1}$. 
For the statement, let us remind that the drift and diffusion coefficients $b$ and $\sigma$ are assumed to be of class $\cinf_b$. 
\begin{lem}\label{lemma19}
For all $K \in \mathbb{N}$ we have:
\begin{equation} \label{ll}
F_n-F_{n-1}=\sum_{k=1}^K \sum_{\alpha \in \mathcal{A}_k} C_1(\alpha, k) J_k^{\alpha}(t_n,t,x)+
\sum_{\alpha \in \mathcal{B}_K} C_2(\alpha, K) R_K^{\alpha}(t_n,t,x),
\end{equation}
where $C_1(\al,k)$ and $C_2(\al,K)$ denote some positive constants.
\end{lem}

\begin{proof}
We will use an induction argument with respect to $K$. First, the case $K=1$ follows applying the mean value theorem to the function $\sigma$ around the point $u_{n-1}$ in the first term on the right-hand side of (\ref{aauu}). 

In order to illustrate our argument to tackle the general case, let us prove our statement for the case $K=2$. In fact, consider the decomposition (\ref{aauu}) and observe that it suffices to prove that the residue of order $1$ (\ref{res1a}) can be written 
as a sum of a process of order $2$ and residues of order $2$. This is achieved by integrating by parts in the integral with respect to $d\lam$:
\begin{equation} \label{rhs}
\begin{split}
&\int_{t_{n-1}}^{t_n} \int_{\R^d} \Gamma(t-s, x-y) \biggl(\int_0^1 \sigma'(u(\lambda, s,y))\, d \lambda\biggr)
(u(s,y)-u_{n-1}(s,y)) W(ds, dy) \\
&=\int_{t_{n-1}}^{t_n} \int_{\R^d} \Gamma(t-s, x-y) \sigma'(u_{n-1}(s,y))
(u(s,y)-u_{n-1}(s,y)) W(ds, dy) \\
&+\int_{t_{n-1}}^{t_n} \int_{\R^d} \Gamma(t-s, x-y) \biggl(\int_0^1 (1-\lambda) \sigma^{(2)}(u(\lambda, s,y))\, d \lambda\biggr)
(u(s,y)-u_{n-1}(s,y))^2 W(ds, dy).
\end{split}
\end{equation}
Note that, one the one hand, by definition the second term on the right-hand side of (\ref{rhs}) is a residue of order $2$. On the other hand, applying the mean value theorem inside the stochastic integral in (\ref{difu}), one obtains the following decomposition:
\begin{equation} \label{difu2}
\begin{split}
u(s,y)-u_{n-1}(s,y)=J_1(s,s,y)+\sum_{\alpha \in \mathcal{B}_1} R_1^{\alpha}(s,s,y).
\end{split}
\end{equation}
Hence, the first term on the right-hand side of (\ref{rhs}) can be written as a sum of a process of order $2$ and two residues of order $2$. This concludes the proof of (\ref{ll}) for $K=2$. We point out that the same arguments used in the latter final part of the case $K=2$ would let us conclude that the residue of order 2 (\ref{r2b}) can be decomposed as the sum of a process of order $3$ and residues of order $3$. 

Now, we assume that (\ref{ll}) holds for $K$, and that any residue of order $k<K$ can be decomposed as a sum of processes of order $k+1$ and residues of order $k+1$. Then, we consider any residue term of order $K$ in (\ref{ll}).
If this residue of order $K$ is of the form (\ref{Rk1}), by integrating by parts one can rewrite it as
\begin{equation*} 
\begin{split}
&\frac{1}{K !} \int_{t_{n-1}}^{t_n} \int_{\R^d} \Gamma(t-s, x-y) \sigma^{(K)}(u_{n-1}(s,y))
(u(s,y)-u_{n-1}(s,y))^K W(ds, dy) \\
&+\frac{1}{K !}
\int_{t_{n-1}}^{t_n} \int_{\R^d} \Gamma(t-s, x-y) \biggl(\int_0^1 (1-\lambda)^K \sigma^{(K+1)}(u(\lambda, s,y))\, d \lambda\biggr) \\
& \qquad \times 
(u(s,y)-u_{n-1}(s,y))^{K+1} W(ds, dy).
\end{split}
\end{equation*}
Thus, by (\ref{difu2}), the above expression can be written as a sum of processes of order $k+1$
and residues of order $k+1$. The same computations can be done if the residue of order $K$ is of the form (\ref{Rk3}). 
Eventually, in case the residue is
of the form (\ref{Rk2}) or (\ref{Rk4}), we will make use of the induction hypothesis. Namely, we will be able to write any $R_{m_j}(K,K,y)$ therein as a sum of processes of order
$m_j+1$ and residues of order $m_j+1$. Therefore, the resulting terms will be sums of processes of order $k+1$ and residues of order $k+1$, which yields that the desired decomposition holds for $K+1$. This concludes the proof.
\end{proof}

\medskip

At this point, we are in position to define the approximation sequence  $\overline{F}_n=\overline{F}_n^K$ needed in hypothesis {\bf{(H1)}}. For all $K\in \mathbb{N}$, set:
\begin{equation*}
\overline{F}_n = \Delta_{n-1} (g)^{\frac{K+1}{2}} Z_n+ F_{n-1}+\sum_{k=1}^K \sum_{\alpha \in \mathcal{A}_k} C_1(\alpha, k) J_k^{\alpha}(t_n,t,x),
\end{equation*}
where the last term above is the one given in Lemma \ref{lemma19}.
Note that here we are dealing with $1$-dimensional random vectors, that is random variables. According to the expression of $\overline{F}_n$ in assumption {\bf{(H1)}}, we have 
\[
h(s,y)=\Gamma(t-s, x-y) \sigma(u_{n-1}(s,y)) \quad \text{and} \quad G_n^K=\sum_{k=2}^K \sum_{\alpha \in \mathcal{A}_k} C_1(\alpha, k) J_k^{\alpha}(t_n,t,x),
\]
and we have taken $\gam=\frac 12$.
The boundedness of $\sig$ and the definition of the processes $J_k^\al$ guarantee that all the conditions of hypothesis {\bf (H1)} are satisfied. Let us also remind that $\{Z_n,\, n\in \mathbb{N}\}$ is an i.i.d. sequence of standard Gaussian random variables which is independent of the noise $W$ (see the explanation right after hypothesis {\bf{(H2d)}}), and that here we consider $g(s)=\Gam(t-s)$, so that
\[
\Del_{n-1}(g)=\int_{t_{n-1}}^{t_n} \int_{\red} |\tf \Gam(t-s)(\xi)|^2 \,\mu(d\xi) ds.
\]

\medskip

We next verify condition {\bf (H2a)}. For this, we observe that Lemma \ref{lemma19} yields:
$$
F_n-\overline{F}_n = -\Delta_{n-1} (g)^{\frac{K+1}{2}} Z_n+\sum_{\alpha \in \mathcal{B}_K} C_2(\alpha, K) R_K^{\alpha}(t_n,t,x).
$$
Then, hypothesis {\bf (H2a)} is a consequence of the following lemma:
\begin{lem}
For all $m \in \mathbb{N}$ and $p>1$, there exists a constant $C$ such that, for any $K \in \mathbb{N}$, we have:
 \begin{equation}\label{eq:101}
\Vert R_K(t_n,t,x) \Vert^{t_{n-1},t_n}_{m, p}\leq C\, \Delta_{n-1} (g)^{\frac{K+1}{2}} \qquad \text{a.s.}
\end{equation}
where $R_K$ denotes any of the four types of residues of order $K$.
\end{lem}

\begin{proof}
We start by assuming that the residue process $R_K$ is of the form (\ref{Rk1}). Then, we will apply Lemma \ref{lema14} in the following setting: $X_0=0$, $f=0$,
$g(z)=\sigma^{(K)}(\lambda z +(1-\lambda) u_{n-1}(r,y))$,
$I_i(r,y)=u(r,y)-u_{n-1}(r,y)$ for all $i=1,\dots,i_0=K$ and $v=u$. Note that Lemmas \ref{lema4} and \ref{lema8} imply that conditions (\ref{lema14a}) and (\ref{lema14b}) are satisfied with $\gamma=\frac12$, $\alpha_i=1$, $\alpha=K$, respectively (we are again making an abuse of notation since the latter $\al$ has nothing to do with the one in $R_K^\al$ above).
Altogether, we obtain that:
\begin{equation*} \begin{split}
 \Vert R_K(t_n,t,x) \Vert^{t_{n-1},t_n}_{m, p}&\leq C \left(\int_{t_{n-1}}^{t_n} J(t-r)\, dr\right)^{\frac 12 -\frac 1p}\\
&\qquad \times\left\{ \int_{t_{n-1}}^{t_n} 
\left( \int_{t_{n-1}}^s J(t-r)\, dr\right)^{\frac{K p}{2}} J(t-s)\, ds \right\}^{\frac 1p} \\
&\leq C\,  \Delta_{n-1} (g)^{\frac{K+1}{2}}, \qquad \text{a.s.},
\end{split}
\end{equation*}
where we have used the fact that $J(r)=\int_{\red} |\tf \Gam(r)(\xi)|^2 \mu(d\xi)$.

Let us now assume that the residue $R_K$ is of the form (\ref{Rk3}). In this case, we apply Lemma \ref{lema14} in the following situation: $X_0=0$, $g=0$, $f(z)=\sigma^{(K-1)}(\lambda z +(1-\lambda) u_{n-1}(r,y))$,
$I_i(r,y)=u(r,y)-u_{n-1}(r,y)$ for all $i=1,\dots,i_0=K-1$, $v=u$, $\gamma=\frac12$, $\alpha_i=1$, $\alpha=K-1$. Thus, we end up with:
\begin{equation*} \begin{split}
 \Vert R_K(t_n,t,x) \Vert^{t_{n-1},t_n}_{m, p}&\leq C \int_{t_{n-1}}^{t_n} \left(\int_{t_{n-1}}^{s} J(t-r)\, dr\right)^{\frac{K-1}{2}} ds  \\
& \leq C \, (t_n-t_{n-1}) \Delta_{n-1} (g)^{\frac{K-1}{2}} \\
&\leq C \, \Delta_{n-1} (g)^{\frac{K+1}{2}}, \qquad \text{a.s.}
\end{split}
\end{equation*}
In the last inequality we have used the estimate (\ref{eq:112}).

We finally show estimate (\ref{eq:101}) for residues $R_K$ of the form (\ref{Rk2}) (the case (\ref{Rk4}) follows along the same lines). This will be deduced thanks to an induction argument on $K$ together with an application of Lemma \ref{lema14}.
We have already checked the validity of (\ref{eq:101}) for $K=1,2$. Assume that the residues (\ref{Rk2}) of order $k\leq K-1$ satisfy the desired estimate. Then, taking into account that in (\ref{Rk2}) there is at least one residue of order $k<K$ hidden in the product of $I_{m_j}(r,y)$ and Lemma \ref{js} below establishes suitable bounds for processes of order $K$, we can conclude by applying again Lemma \ref{lema14} as follows: $X_0=0$, $f=0$,
$g(z)=\sigma^{(K)}(z)$, $v=u_{n-1}$, $\gamma=\frac12$, $\alpha_i=m_j$, $\alpha=K$. Details are left to the reader. 
\end{proof}

\medskip

A consequence of the above lemma is that condition {\bf{(H2a)}} is satisfied. Moreover, as we have already explained at the beginning of the present section, hypothesis {\bf{(H2b)}}  follows from Proposition \ref{prop1}. Next we prove {\bf{(H2c)}}, that is 
\beq \label{eq:110}
C_1 \leq \Del_{n-1}(g)^{-1} \int_{t_{n-1}}^{t_n} \|h(s)\|^2_\hac \, ds \leq C_2,
\eeq
for some positive constants $C_1, C_2$, where we recall that $h(s,y)=\Gam(t-s,x-y)\sig(u_{n-1}(s,y))$. The upper bound is an immediate consequence of \cite[Theorem 2]{Dalang:99}, since we assume that $\sig$ is bounded. In order to obtain the lower bound  in (\ref{eq:110}), we can use the same argument of \cite[Theorem 5.2]{Nualart:07}. Precisely, let $(\psi_\ep)_\ep$ be an approximation of the identity and define $G_\ep(s,y):=(\psi_\ep * h(s))(y)$, so that $G_\ep(s)\in \mathcal{S}(\red)$. Then, by the non-degeneracy assumption on $\sig$, we can infer that, almost surely:
\begin{align*}
\int_{t_{n-1}}^{t_n} \|h(s)\|^2_\hac \, ds & = \lim_{\ep\rightarrow 0} \int_{t_{n-1}}^{t_n} \|G_\ep(s)\|^2_\hac \, ds \\
& = \lim_{\ep\rightarrow 0} \int_{t_{n-1}}^{t_n} ds \int_{\red} \Lam(dy) \int_{\red} dz \, G_\ep(s,z) G_\ep(s,z-y) \\
& \geq C\, \int_{t_{n-1}}^{t_n} \int_{\red} |\tf \Gam(t-s)(\xi)|^2 \, \mu(d\xi) ds\\
& = C \, \Del_{n-1} (g). 
\end{align*} 
Therefore, we obtain the lower bound in (\ref{eq:110}), which means that {\bf{(H2c)}} is satisfied.

\medskip

Eventually, we check condition {\bf (H2d)}.  To begin with, let us remind that 
\begin{equation*}
G_n^K =\sum_{k=2}^K \sum_{\alpha \in \mathcal{A}_k} C_1(\alpha, k) J_k^{\alpha}(t_n,t,x).
\end{equation*}
Then, hypothesis {\bf (H2d)} is proved in the following lemma:
\begin{lem} \label{js}
For all $m \in \mathbb{N}$ and $p>1$, there exists a constant $C$ such that, for any $k \in \mathbb{N}$, we have:
 \begin{equation*}
\Vert J_k(t_n,t,x) \Vert^{t_{n-1},t_n}_{m, p}\leq C\, \Delta_{n-1} (g)^{\frac{k}{2}} \qquad \text{a.s.},
\end{equation*}
where $J_k$ denotes an arbitrary process of order $k$.
\end{lem}

\begin{proof} It is based on several applications of Lemma \ref{lema14}. We shall perform an induction argument with respect to $k$. Assume that $J_k$ is of the form (\ref{Jk1}). 

The case $k=1$ follows directly from Lemma \ref{lema14} applied to the case $X_0=0$, $f=0$,
$g(z)=\sigma(z)$, $v=u_{n-1}$ and $\gamma=0$, where the constant on the right-hand side of (\ref{lema14a}) only depends on $j,p$ and $T$. Similarly, the case $k=2$ follows again appealing to Lemma \ref{lema14}, with $X_0=0$, $f=0$,
$g(z)=\sigma'(z)$, $v=u_{n-1}$, $i_0=1$, $\gamma=\frac12$, $\alpha=1$, and using the result for $k=1$.
 
Assume that the statement holds for any process of order up to $k-1$. Then, if $J_k$ is given by (\ref{Jk1}), we can apply Lemma \ref{lema14} in the case $X_0=0$, $f=0$,
$g(z)=\sigma^{(\ell)}(z)$, $v=u_{n-1}$, $i_0=\ell$, $\gamma=\frac12$, $\alpha_i=m_i$ and $\alpha=k-1$. 
Observe that the induction hypothesis guarantees that condition (\ref{lema14b}) is satisfied. Altogether, we obtain:
\begin{equation*} \begin{split}
 \Vert J_k(t_n,t,x) \Vert^{t_{n-1},t_n}_{m, p}&\leq C 
\left(\int_{t_{n-1}}^{t_n} J(t-r)\, dr\right)^{\frac 12 -\frac 1p}\\
&\qquad \times  \left\{ \int_{t_{n-1}}^{t_n} 
\left( \int_{t_{n-1}}^s J(t-r)\, dr\right)^{\frac{(k-1) p}{2}} J(t-s)\, ds \right\}^{\frac 1p} \\
&\leq C \, \Delta_{n-1} (g)^{\frac{k}{2}}, \qquad \text{a.s.}
\end{split}
\end{equation*}

On the other hand, if $J_k$ is of the form (\ref{Jk1b}), one proceeds similarly as before. Namely, we apply  Lemma \ref{lema14} with $X_0=0$, $g=0$,
$f(z)=\sigma^{(\ell)}(z)$, $v=u_{n-1}$, $i_0=\ell$,$\gamma=\frac12$, $\alpha_i=m_i$, $\alpha=k-2$, so that we end up with:
\begin{equation*} \begin{split}
 \Vert J_k(t_n,t,x) \Vert^{t_{n-1},t_n}_{m, p}&\leq C \int_{t_{n-1}}^{t_n} \left(\int_{t_{n-1}}^{s} J(t-r)\, dr\right)^{\frac{k-2}{2}} ds  \\
& \leq C\, (t_n-t_{n-1}) \Delta_{n-1} (g)^{\frac{k-2}{2}} \\
&\leq C\, \Delta_{n-1} (g)^{\frac{k}{2}}, \qquad \text{a.s.},
\end{split}
\end{equation*}
where the last inequality follows from (\ref{eq:112}). This concludes the proof.
\end{proof}

\medskip

With this lemma, we conclude that hypotheses {\bf{(H1)}}, {\bf{(H2a)}}-{\bf{(H2d)}} are satisfied and we obtain, in view of Definition \ref{def}, that the random variable $F=u(t,x)$ is uniformly elliptic. Therefore, by Theorem \ref{t1}, we have proved the lower bound in Theorem \ref{maint}.

\section{Proof of the upper bound}
\label{sec:upper-bound}

This section is devoted to prove the upper bound of Theorem \ref{maint}.  For this,
we will follow a standard procedure based on the density formula provided by the integration-by-parts formula of the Malliavin calculus and the exponential martingale inequality applied to the martingale part of our random variable $u(t,x)$ for $(t,x)\in (0,T]\times \red$ 
(this method has been used for instance in \cite[Proposition 2.1.3]{Nualart:06} and \cite{Guerin,Dalang:09}). We remind that we are assuming that the coefficients $b$ and $\sig$ belong to $\mathcal{C}^\infty_b(\re)$ and the spectral measure $\mu$ satisfies 
\[
\int_{\red}\frac{1}{1+\Vert \xi \Vert^2}\, \mu(d\xi) < +\infty.
\]
Moreover, we have that:
\begin{equation} \begin{split}
u(t,x)= F_0&+ \int_0^t \int_{\red} \Gam(t-s,x-y)\sig(u(s,y)) W(ds,dy)  \\
& + \int_0^t \int_{\red} \Gam(t-s,x-y) b(u(s,y)) \, dy ds, \quad a.s.,
\label{eq:222}
\end{split}
\end{equation}
where $F_0=(\Gam(t)*u_0)(x)$.

\medskip

To begin with, we consider the continuous one-parameter martingale $\{M_a, \tf_a,\, 0 \leq a \leq t\}$ defined by
\begin{equation*}
M_a=\int_0^a \int_{\R^d} \Gamma(t-s,x-y) \sigma(u(s,y)) \, W(ds, dy),
\end{equation*}
where the  filtration $\{\mathcal{F}_a, \, 0\leq a \leq t\}$ is the one generated by $W$. Notice that
$M_0=0$ and one has that
\begin{equation*}
\langle M \rangle_t =\Vert \Gamma(t-\cdot,x-\star) \sigma(u(\cdot,\star)) \Vert_{\mathcal{H}_t}.
\end{equation*}
Since $\sigma$ is bounded, we clearly get that $\langle M \rangle_t \leq c_2 \Phi(t)$, a.s. for some positive constant $c_2$ (see for instance \cite[Theorem 2]{Dalang:99}).

On the other hand, since the drift $b$ is also assumed to be bounded and $\Gam(s)$ defines a probability density, we can directly estimate the drift term in (\ref{eq:222}) as follows: 
\begin{equation}
\left\vert\int_0^t \int_{\red} \Gam(t-s,x-y) b(u(s,y)) \, dy ds \right \vert \leq c_3 \, T, \qquad \text{a.s.}
\label{eq:224}
\end{equation}

We next consider the expression for the density of a non-degenerate random variable that follows from the integration-by-parts formula of the Mallavin calculus. Precisely, we apply \cite[Proposition 2.1.1]{Nualart:06} so that we end up with the following expression for the density $p_{t,x}$ of $u(t,x)$:
\[
p_{t,x}(y)=E \left[ {\bf{1}}_{\{ u(t,x)>y \}} \del (Du(t,x)\, \|Du(t,x)\|^{-2}_{\hact} )\right], \quad y\in \re,
\]
where $\del$ denotes the divergence operator or Skorohod integral, that is the adjoint of the Malliavin derivative operator (see \cite[Ch. 1]{Nualart:06}).
Taking into account that the Skorohod integral above has mean zero, one can also check that:
\[
p_{t,x}(y)= -E \left[ {\bf{1}}_{\{ u(t,x)<y \}} \del (Du(t,x)\, \|Du(t,x)\|^{-2}_{\hact} )\right], \quad y\in \re.
\]
Then, owing to (\ref{eq:222}), \cite[Proposition 2.1.2]{Nualart:06} and the estimate (\ref{eq:224}), we can infer that:
\begin{equation}
\begin{split}
p_{t,x}(y) &\leq c_{\alpha, \beta, q} \P\{|M_t| > |y-F_0|-c_3T\}^{1/q} \\
& \qquad \times \biggl( \E[\Vert D u(t,x) \Vert_{\mathcal{H}_t}^{-1}]+ \Vert D^2 u(t,x)\Vert_{L^{\alpha}(\Omega; \mathcal{H}_t^{\otimes 2})} \Vert \Vert D u(t,x) \Vert^{-2}_{\mathcal{H}_t} \Vert_{ L^\beta(\Om)}\biggr),
\end{split}
\label{eq:223}
\end{equation}
where $\alpha, \beta, q$ are any positive real numbers satisfying $\frac{1}{\alpha}+\frac{1}{\beta}+\frac{1}{q}=1$.
Thus, we proceed to bound all the terms on the right-hand side of (\ref{eq:223}). 

First, by the exponential martingale inequality (see for instance \cite[Sec. A2]{Nualart:06}) and the fact that $\langle M \rangle_t \leq c_2 \Phi(t)$, we obtain:
\begin{equation} 
\P \{|M_t| > \vert y-F_0\vert-c_2T \} \leq 2 \exp \biggl( -\frac{(\vert y-F_0 \vert-c_3T)^2}{ c_2 \Phi(t)}\biggr).
\label{eq:226}
\end{equation}
Secondly, we observe that the following estimate is satisfied: for all $p>0$, there exists a constant $C$, depending also on $T$, such that
\beq
E( \|Du(t,x)\|_{\hac_t}^{-2p} )\leq C\, \Phi(t)^{-p}.
\label{eq:225}
\eeq
\begin{remark}
Indeed, this bound could be considered as a kind of particular case of (\ref{aux}). Nevertheless, though its proof is very similar to that of Proposition \ref{prop1} and we will omit it (an even more similar proof is given in \cite[Theorem 6.2]{Nualart:07}), it is important to make the following observation. Namely, the fact that in (\ref{eq:225}), compared to Proposition \ref{prop1}, we are not considering a time interval of the form $[t_{n-1},t_n]$ but directly $[0,t]$,  makes it possible to obtain the desired estimate (\ref{eq:225}) without using condition (\ref{eq:112}). 
Therefore, owing to Remark \ref{rmk:heat}, one could verify that the density upper bound we are proving in the present section turns out to be valid for the stochastic wave equation in space dimension $d\in \{1,2,3\}$, with the only differences that $F_0$ and the term $c_3 T$ should be replaced by the corresponding contribution of the initial conditions and $c_3T^2$, respectively. Note that the power $T^2$ in that case is due to the total mass of the associated fundamental solution considered as a measure in $\red$ (see \cite[Example 6]{Dalang:99}). Finally, we point out that, for spatial dimension one, density upper bounds for a reduced stochastic wave equation have been obtained in \cite{Dalang:04}.
\end{remark}

Going back to our estimate (\ref{eq:225}), let in particular $p=\frac 12$ and $p=\beta$ in such a way that, respectively:
\begin{equation}
\E[\Vert D u(t,x) \Vert_{\mathcal{H}_t}^{-1}] \leq C\, \Phi(t)^{-1/2} \quad \text{and}\quad 
\Vert \Vert D u(t,x) \Vert^{-2}_{\mathcal{H}_t} \Vert_{L^\beta(\Om)} \leq C\, \Phi(t)^{-1}.
\label{eq:227}
\end{equation}

Eventually, Lemma \ref{lema4} implies that
\begin{equation} 
\Vert D^2 u(t,x)\Vert_{L^{\alpha}(\Omega; \mathcal{H}_t^{\otimes 2})} \leq C\, \Phi(t) \leq C\, \Phi(t)^{\frac 12},
\label{eq:228}
\end{equation}
where the latter constant $C$ depends on $T$. 
Hence, plugging estimates (\ref{eq:226})-(\ref{eq:228}) into (\ref{eq:223}) we end up with:
\begin{equation*}
p_{t,x}(y) \leq c_1\, \Phi(t)^{-1/2} \exp \biggl( -\frac{(|y-F_0\vert-c_3 T)^2}{c_2 \Phi(t)}\biggr),
\end{equation*}
where the constants $c_i$ do not depend on $(t,x)$. This concludes the proof of Theorem \ref{maint}.

\appendix
\section{Appendix}

The first part of this section is devoted to recall the construction of the Hilbert-space-valued stochastic and pathwise integrals used throughout the paper, as well as establish the corresponding conditional $L^p$-bounds for them. In the second part, we state and proof a technical result that has been very useful in the proofs of Section \ref{sec:lower-bound}.

\subsection{Hilbert-space-valued stochastic and pathwise integrals}
\label{Hilbert-integrals}

Let us briefly explain the construction of stochastic and pathwise integrals in a Hilbert-space-valued setting, the former of course being with respect to $W$. This is an important point in order to consider the linear stochastic equation satisfied by the iterated Malliavin derivative of the solution of many SPDEs (for a more detailed exposition, see \cite[Section 2]{Quer:04} and \cite{Marta}).

More precisely, let $\ca$ be a separable Hilbert space and
$\{K(t,x),\; (t,x)\in [0,T]\times\red\}$ an $\ca$-valued predictable process satisfying the following condition:
\beq
\sup_{(t,x)\in [0,T]\times\red} E(\|K(t,x)\|^p_\ca)<+\infty,
\label{eq:1}
\eeq
where $p\geq 2$. We aim to define the $\ca$-valued stochastic integral
\[
G\cdot W_t=\int_0^t \int_{\red} G(s,y) W(ds,dy), \quad t\in [0,T],
\]
for integrands of the form $G=\Gam(s,dy)K(s,y)$, where we assume here that $\Gam$ is as described in Remark \ref{rmk:general-gam}. In particular, $\Gam$ satisfies condition (\ref{hyp1}), that is:
\[
\int_0^T \int_{\R^d} \vert \mathcal{F}\Gamma(t) (\xi) \vert^2 \mu(d\xi) dt<+\infty.
\]
Note that these assumptions imply that $G$ is a well-defined element of $L^2(\Om\times [0,T]; \hac\otimes \ca)$. Recall that we denote by $\{\tf_t,\; t\geq 0\}$ the (completed) filtration generated by $W$.
Then, the stochastic integral of $G$ with respect to $W$ can be defined componentwise, as follows: let $\{e_j,\; j\in \N\}$ be a complete orthonormal basis of $\ca$ and set $G^j:=\Gam(s,dy)K^j(s,y)$, where $K^j(s,y):=\langle K(s,y),e_j\rangle_\ca$, $j\in \N$. We define
\[
G\cdot W_t:= \sum_{j\in \N} G^j\cdot W_t,
\]
where $G^j\cdot W_t=\int_0^t\int_{\red} \Gam(s,y)K^j(s,y)W(ds,dy)$ is a well-defined real-valued stochastic integral (see \cite[Remark 1]{Quer:04}). By (\ref{eq:1}), one proves that the above series is convergent in $L^2(\Om;\ca)$ and the limit does not depend on the orthonormal basis. Moreover, $\{G\cdot W_t, \tf_t, \; t\in[0,T]\}$ is a continuous square-integrable martingale such that
\[
E(\|G\cdot W_T\|^2_\ca)=E(\|G\|^2_{\hac_T\otimes \ca}).
\]
We also have the following estimate for the $p$th moment of $G\cdot W_t$ (see \cite[Theorem 1]{Quer:04}): for all $t\in [0,T]$,
\beq
E(\|G\cdot W_t\|^p_\ca)\leq C_p \Phi(t)^{\frac p2-1} \int_0^t \sup_{x\in \red} E(\|K(s,x)\|^p_\ca) J(s) \,ds,
\label{eq:2}
\eeq
where we remind that
\[
\Phi(t)= \int_0^t \int_{\red} |\tf \Gam(s)(\xi)|^2 \, \mu(d\xi)ds \quad \text{and} \quad J(s)=\int_{\red} |\tf \Gam(s)(\xi)|^2 \, \mu(d\xi).
\]

Next, we consider a conditional version of (\ref{eq:2}):
\begin{lem} \label{lema2}
For all $p\geq 2$ and $0\leq a<b\leq T$, we have:
\[
E [\|G\cdot W_b-G\cdot W_a\|^p_\ca| \tf_a]\leq C_p (\Phi(b)-\Phi(a))^{\frac p2-1} \int_a^b \sup_{x\in \red} E[\|K(s,x)\|^p_\ca|\tf_a] \,J(s) \,ds, \quad a.s.
\]
\end{lem}
The proof of this result is essentially the same as its non-conditioned counterpart (\ref{eq:2}), except of the use of
a conditional Burkholder-Davis-Gundy type inequality for Hilbert-space-valued martingales.

\medskip

Let us now recall how we define the Hilbert-space-valued pathwise integrals involved in the stochastic equations satisfied by the Malliavin derivative of the solution. Namely, as before, we consider a Hilbert space $\ca$, a complete orthonormal system $\{e_j,\; j\in \N\}$, and an $\ca$-valued stochastic process
$\{Y(t,x),\; (t,x)\in [0,T]\times \red\}$ such that, for $p\geq 2$,
\beq
\sup_{(t,x)\in [0,T]\times \red} E(\|Y(t,x)\|^p_\ca)<+\infty.
\label{eq:4}
\eeq
Then, we define the following pathwise integral, with values in $L^2(\Om;\ca)$:
\[
\ci_t:=\int_0^t \int_{\red} Y(s,y)\, \Gam(s,dy) ds:=\sum_{j\in \N} \left( \int_0^t \int_{\red} \langle Y(s,y),e_j\rangle_\ca \, \Gam(s,dy) ds\right) e_j, \quad t\in [0,T],
\]
where $\Gam$ is again as general as described in Remark \ref{rmk:general-gam}. Moreover, a direct consequence of the considerations in \cite[p. 24]{quer-thesis} is that:
\beq
E(\|\ci_t\|^p_\ca)\leq \left(\int_0^t \Gam(s,\red)\, ds\right)^{p-1} \int_0^t \sup_{z\in \red} E(\|Y(s,z)\|^p_\ca) \,\Gam(s,\red)\, ds.
\label{eq:3}
\eeq
In the paper, we need the following straightforward conditional version of the above estimate (\ref{eq:3}):
\begin{lem} \label{lema3}
Let $p\geq 2$. Then, for any $\sig$-field $\cg$, we have:
\[
E[\|\ci_t\|^p_\ca | \cg] \leq \left(\int_0^t \Gam(s,\red)\, ds\right)^{p-1} \int_0^t \sup_{z\in \red} E[\|Y(s,z)\|^p_\ca |\cg] \,\Gam(s,\red)\, ds, \quad a.s.
\]
\end{lem}

\subsection{An auxiliary result}
\label{sec:aux}

Let $a\leq b\leq t$ and $x\in \red$, and consider the following random variable:
\begin{equation} \begin{split}
X = X_0(t,a,b)&+\int_a^b \int_{\red} \Gam(t-s,x-y) f(v(s,y)) \prod_{i=1}^{i_0} I_i(s,y) \, dyds \\
& + \int_a^b \int_{\red} \Gam(t-s,x-y) g(v(s,y)) \prod_{i=1}^{i_0} I_i(s,y) W(ds,dy),
\label{eq:44}
\end{split}
\end{equation}
where $X_0(t,a,b)$ is a $\tf_a$-measurable random variable, $f,g:\re\rightarrow \re$ are smooth functions and $v$ and $I_i$ are certain smooth stochastic processes (in the Malliavin sense). As usual, $\Gam$ denotes the fundamental solution of the stochastic heat equation in $\red$. In the next result, we provide an estimate for the $p$th moment of the iterated Malliavin derivative of order $m$ for the random variable $X$.

\begin{lem}\label{lema14}
Assume that $f,g\in \cinf_b(\re)$ and that $v$ and $I_1,\dots,I_{i_0}$ are smooth stochastic processes on $[a,b]$ for which the random variable \textnormal{(\ref{eq:44})} is well-defined. Let $m\in \N$ and $p\geq 2$. Suppose that, for all $j=0,\dots,m$ and $s\in [a,t]$, there exist constants $c(j,p,a,s)>1$ and
$C(m,p)$ such that the former increases in $p$ and
\begin{equation} \label{lema14a}
\left( E_a \|D^j v(s,y)\|^p_{\hac_{a,b}^{\otimes j}} \right)^{\frac 1p} \leq c(j,p,a,s) \quad \text{and}
\end{equation}
\begin{equation} \label{lema14b}
\left( E_a \|D^j I_i(s,y)\|^p_{\hac_{a,b}^{\otimes j}} \right)^{\frac 1p} \leq C(m,p) \left(\int_a^s J(t-r)\, dr\right)^{\gam \al_i},
\end{equation}
for all $i=1,\dots,i_0$ and $y\in \red$, and some $\gam\geq 0$, $\al_1,\dots,\al_{i_0}>0$, where we recall that $J(t-r)=\int_{\red} |\tf \Gam(t-r)(\xi)|^2 \mu(d\xi)$.

Then $X\in \D^{m,\infty}$ and the following is satisfied: if we set $\al:=\al_1+\dots+\al_{i_0}$, there exist $p'$ and $p^*$ such that
\begin{equation} \begin{split}
&\left( E_a \|D^m X\|^p_{\hac_{a,b}^{\otimes m}}\right)^{\frac 1p} \leq \left( E_a \|D^m X_0(t,a,b)\|^p_{\hac_{a,b}^{\otimes m}}\right)^{\frac 1p}\\
&\qquad \qquad+ C_1 \int_a^b \left[ c^*(m-1,p',a,s)^m+c(m,p^*,a,s)\right] \left(\int_a^s J(t-r)dr\right)^{\al \gam} ds \\
& \qquad\qquad+ C_2 \left(\int_a^b J(t-r)\, dr\right)^{\frac 12 -\frac 1p} \left\{ \int_a^b \left[ c^*(m-1,p',a,s)^m+c(m,p^*,a,s)\right]^p \right. \\
& \qquad\qquad\qquad \times \left. \left( \int_a^s J(t-r)\, dr\right)^{\al \gam p} J(t-s)\, ds \right\}^{\frac 1p},
\label{eq:45}
\end{split}
\end{equation}
where $C_1, C_2$ are some positive constants possibly depending on $m$ and $p$, such that if $f \equiv 0$ then $C_1=0$, and if
$g \equiv 0$ then $C_2=0$.
We also use the notation $c^*(m-1,p',a,s):=\max_{0\leq j\leq m-1} c(j,p',a,s)$ and we set $c(-1,p,a,s)=0$.

In the case where $c(j,p',a,s)<1$ for all $j=1,\dots,m$, estimate \textnormal{(\ref{eq:45})} is replaced by
\begin{align*}
\left( E_a \|D^m X\|^p_{\hac_{a,b}^{\otimes m}}\right)^{\frac 1p} &\leq \left( E_a \|D^m X_0(t,a,b)\|^p_{\hac_{a,b}^{\otimes m}}\right)^{\frac 1p}
\nonumber \\
& \quad + C_1 \int_a^b \left[ c^*(m-1,p',a,s)+c(m,p^*,a,s)\right] \left(\int_a^s J(t-r)dr\right)^{\al \gam} ds \nonumber \\
& \quad + C_2 \left(\int_a^b J(t-r)\, dr\right)^{\frac 12 -\frac 1p} \left\{ \int_a^b \left[ c^*(m-1,p',a,s)+c(m,p^*,a,s)\right]^p \right. \nonumber \\
& \qquad \quad \times \left. \left( \int_a^s J(t-r)\, dr\right)^{\al \gam p} J(t-s)\, ds \right\}^{\frac 1p}.
\end{align*}

\end{lem}

\begin{proof}
The assumptions on the functions $f,g$ and the processes $v, I_i$ clearly yield that $X\in \D^{m,\infty}$. Hence, the proof will be devoted to establish estimate (\ref{eq:45}). The method is similar to that of Lemma 14 in \cite{Kohatsu:03}, and thus we will only focus on the parts of the proof which really exhibit a different methodology.

To begin with, we have that
\begin{equation} \begin{split}
 & \left( E_a \|D^m X\|^p_{\hac_{a,b}^{\otimes m}} \right)^{\frac 1p} \leq \left( E_a \|D^m X_0(t,a,b)\|^p_{\hac_{a,b}^{\otimes m}}\right)^{\frac 1p} \\
& \; + \int_a^b\int_{\red} \Gam(t-s,x-y)
\left( E_a \|D^m f(v(s,x-y)) \prod_{i=1}^{i_0} I_i(s,x-y)\|^p_{\hac_{a,b}^{\otimes m}} \right)^{\frac 1p} \, dyds \\
& \; + \left( E_a \left\| D^m \left(\int_a^b\int_{\red} \Gam(t-s,x-y) g(v(s,y)) \prod_{i=1}^{i_0} I_i(s,y) W(ds,dy)\right) \right\|^p_{\hac_{a,b}^{\otimes m}} \right)^{\frac 1p}
\label{eq:46}
\end{split}
\end{equation}
The second term on the right-hand side of the above inequality may be bounded in the same manner as for the corresponding term in the proof of \cite[Lemma 14]{Kohatsu:03} (see (13) therein). Indeed, one makes use of the Leibniz rule for the iterated Malliavin derivative and applies a generalization of H\"older's inequality, which altogether yields to
\begin{equation}
\begin{split} 
& \int_a^b\int_{\red} \Gam(t-s,x-y) \left( E_a \|D^m f(v(s,x-y)) \prod_{i=1}^{i_0} I_i(s,x-y)\|^p_{\hac_{a,b}^{\otimes m}} \right)^{\frac 1p} \, dyds \\
& \; \leq C(m,p) \int_a^b \left( c^*(m-1,p',a,s)^m + c(m,p^*,a,s)\right) \left( \int_a^s J(t-r) \, dr\right)^{\al \gam} ds.
\label{eq:47}
\end{split}
\end{equation}
Here, we have $p':=\max_k p_k$, where this maximum is taken over a finite set whose cardinal depends on $m$, and $p_k>0$ denote a certain set of real numbers needed for the application of the above-mentioned H\"older's inequality (see \cite[p. 455]{Kohatsu:03} for details). On the other hand, $p^*:=p$ in the case where $I_i$ is constant, for all $i$; otherwise $p^*:=p'$. Furthermore, in the case where $c(j,p',a,s)<1$ for all $j=1,\dots,m$, estimate (\ref{eq:47}) turns out to be
\[
C(m,p)\int_a^b \left( c^*(m-1,p',a,s) + c(m,p^*,a,s)\right) \left( \int_a^s J(t-r) \, dr\right)^{\al \gam} ds.
\]

In order to deal with the last term on the right-hand side of (\ref{eq:46}), the computations slightly differ from those used in \cite[Lemma 14]{Kohatsu:03}, since we are
considering the more general setting determined by the Hilbert space $\hac$. We will use the following notation: $\bar s:=(s_1,\dots,s_m)\in [a,b]^m$, and
$\bar s(j):=(s_1,\dots,s_{j-1},s_{j+1},\dots,s_m)$. Then, for instance, for any smooth random variable $Y$, $D^m_{\bar s}Y$ denotes the $\hac^{\otimes m}$-valued random variable defined by $(D^mY)(\bar s,\star)$.

Using this notation, we can first infer that, for all $\bar s\in [a,b]^m$:
\begin{equation}
\begin{split}
& D^m_{\bar s} \left( \int_a^b\int_{\red} \Gam(t-s,x-y) g(v(s,y)) \prod_{i=0}^{i_0} I_i(s,y) W(ds,dy) \right) \nonumber \\
& \quad = \sum_{j=1}^m \Gam(t-s_j,x-\star) D^{m-1}_{\bar s(j)} \left( g(v(s_j,\star)) \prod_{i=0}^{i_0} I_i(s_j,\star) \right) \\
& \quad \; + \int_{a\vee s_1\vee\dots \vee s_m}^b \int_{\red} \Gam(t-s,x-y) D^m_{\bar s}\left( g(v(s,y)) \prod_{i=0}^{i_0} I_i(s,y)\right) W(ds,dy),
\label{eq:48}
\end{split}
\end{equation}
where these equalities are understood as random variables with values in $\hac^{\otimes m}$, and we recall that $\star$ denotes the $\hac$-variable. Thus, the last term in (\ref{eq:46}) may be bounded by $A_1+A_2$, with
\begin{equation*} \begin{split}
&A_1^p= E_a \left| \int_{(a,b)^m} \left\| \sum_{j=1}^m \Gam(t-s_j,x-\star) D^{m-1}_{\bar s(j)} \left( g(v(s_j,\star)) \prod_{i=0}^{i_0} I_i(s_j,\star) \right) \right\|^2_{\hac^{\otimes m}} d\bar s\right|^{\frac p2},\\
&A_2^p=E_a \left| \int_{(a,b)^m} \left\| \int_{a\vee s_1\vee\dots \vee s_m}^b \int_{\red} \Gam(t-s,x-y) D^m_{\bar s}\left( g(v(s,y)) \prod_{i=0}^{i_0} I_i(s,y)\right) W(ds,dy) \right\|^2_{\hac^{\otimes m}} 
d\bar s \right|^{\frac p2}.
\end{split}
\end{equation*}
The term $A_1$ can be treated using similar arguments as those of \cite[Lemma 2]{Sanz:04}. Indeed, if we set
\[
Z_j(\bar s,y):= \left\| D^{m-1}_{\bar s(j)} \left( g(v(s_j,y)) \prod_{i=0}^{i_0} I_i(s_j,y) \right) \right\|_{\hac^{\otimes (m-1)}}
\]
(which is a real-valued random variable), then the fact that $\Gam$ is a smooth function and Cauchy-Schwarz and H\"older's inequalities yield, up to some positive constant,
\begin{equation}
 \begin{split}
A_1^p & \leq \sum_{j=1}^m E_a \left| \int_a^b ds_j \int_{\red} \Lam(dy) \int_{\red} dz \; \Gam(t-s_j,x-z) \Gam(t-s_j,x-z+y) \right. \\
& \qquad \times \left( \int_{(a,b)^{m-1}} Z_j(\bar s,z) Z_j(\bar s,z-y) \, d\bar s(j)\right) \Big|^{\frac p2} \\
& \leq \left(\int_a^b J(t-r)\, dr\right)^{\frac p2 -1} \sum_{j=1}^m \int_a^b \sup_{y\in \red} E_a \left( \int_{(a,b)^{m-1}} |Z_j(\bar s,y)|^2\, d\bar s(j) \right)^{\frac p2} J(t-s_j) \, ds_j.
\label{eq:49}
\end{split}
\end{equation}
The $\frac p2$th moment appearing in the latter term can be estimated using the same arguments we have commented above to obtain (\ref{eq:47}). Namely
\begin{align*}
&E_a \left( \int_{(a,b)^{m-1}} |Z_j(\bar s,y)|^2\, d\bar s(j) \right)^{\frac p2}  = E_a \left( \left\| D^{m-1}\left( g(v(s_j,y)) \prod_{i=1}^{i_0} I_i(s_j,y)\right) \right\|^{p}_{\hac^{\otimes (m-1)}_{a,b}} \right)\\
&\qquad \qquad \qquad \qquad\leq C(m-1,p) c^*(m-1,p',a,s_j)^{(m-1)p} \left(\int_a^{s_j} J(t-r)\,dr\right)^{\al \gam p}.
\end{align*}
Plugging this bound in (\ref{eq:49}) we eventually end up with
\begin{equation} \begin{split}
A_1^p & \leq m C(m-1,p) \left(\int_a^b J(t-r)\, dr\right)^{\frac p2 -1} \\
& \qquad \times \int_a^b c^*(m-1,p',a,s)^{(m-1)p} \left(\int_a^s J(t-r)\,dr\right)^{\al \gam p} J(t-s) \, ds.
\label{eq:50}
\end{split}
\end{equation}

In order to bound the term $A_2^p$, let us apply Lemma \ref{lema2} in the particular case where $\ca=\hac^{\otimes m}_{a,b}$ and
\[
G= D^m \left( g(v(s,y)) \prod_{i=0}^{i_0} I_i(s,y)\right).
\]
Hence, we have that, up to some positive constant,
\[
A_2^p \leq \left(\int_a^b J(t-r)\, dr\right)^{\frac p2 -1} \int_a^b \sup_{y\in \red} E_a \left\|D^m \left( g(v(s,y)) \prod_{i=0}^{i_0} I_i(s,y)\right) \right\|_{\hac^{\otimes m}_{a,b}}^p J(t-s)\, ds.
\]
At this point, one applies the same method that we have used to obtain estimate (\ref{eq:47}) (see also the last part in the proof of \cite[Lemma 14]{Kohatsu:03}), so that we can infer that
\begin{equation}
\begin{split}
A_2^p & \leq \left(\int_a^b J(t-r)\, dr\right)^{\frac p2 -1} \int_a^b \left[ c^*(m-1,p',a,s)^m+c(m,p^*,a,s)\right]^p \\
& \qquad \quad \times \left( \int_a^s J(t-r)\, dr\right)^{\al \gam p} J(t-s)\, ds.
\label{eq:51}
\end{split}
\end{equation}
We conclude the proof by putting together estimates (\ref{eq:47}), (\ref{eq:50}) and (\ref{eq:51}).

\end{proof}

\begin{remark}
As for Lemma \ref{lema4}, the above lemma still remains valid for a slightly more general situation. Namely, in the case where $\Gam$ satisfies the assumptions specified in Remark \ref{rmk:general-gam}, such as for the stochastic wave equation in space dimension $d\in \{1,2,3\}$. In such a general setting, the proof of Lemma \ref{lema14} becomes even more technical and tedious since one needs to smooth $\Gam$ by means of an approximation of the identity. For the sake of conciseness, we have decided to focus the proof on our stochastic heat equation.
\end{remark}

\bigskip

\bigskip

\noindent {\bf{Acknowlegments}} 

\noindent This work started while the first author was visiting the Centre de Recerca Matem\`atica (Barcelona), to which she would like to thank for the financial support. Part of this work has also been done while the authors visited the Hausdorff Research Institute for Mathematics (HIM) in Bonn, where they have been supported by a grant in the framework of the HIM Junior Trimester Program on Stochastics. The second author is also supported by the grant MCI-FEDER Ref. MTM2009-08869.


\end{document}